\documentclass[a4paper]{amsart}
\usepackage{amsmath,amsthm,amssymb,latexsym,epic,bbm,comment,yfonts,mathrsfs}
\usepackage{graphicx,enumerate,stmaryrd,rotating}
\usepackage[all]{xy}
\xyoption{2cell}

\allowdisplaybreaks

\usepackage[active]{srcltx}
\usepackage[parfill]{parskip}

\newtheorem{theorem}{Theorem}
\newtheorem{lemma}[theorem]{Lemma}
\newtheorem{remark}[theorem]{Remark}
\newtheorem{schol}[theorem]{Scholium}
\newtheorem{corollary}[theorem]{Corollary}
\newtheorem{proposition}[theorem]{Proposition}
\newtheorem{example}[theorem]{Example}

\newtheorem{definition}[theorem]{Definition}

\numberwithin{equation}{section}

\newcommand{\tto}{\twoheadrightarrow}

\newcommand{\gggg}{{\mathrm{g}}}

\newcommand{\Hom}{{\mathrm{Hom}}}
\newcommand{\Nat}{\mathrm{Nat}}
\newcommand{\SNat}{\mathrm{SNat}}
\newcommand{\Rep}{{\mathrm{Rep}}}
\newcommand{\Ext}{{\mathrm{Ext}}}
\newcommand{\Fun}{{\mathrm{Func}}}

\newcommand{\Ev}{{\mathrm{Ev}}}

\newcommand{\add}{{\mathrm{add}}}

\newcommand{\eins}{\leavevmode\hbox{\small1\kern-3.8pt\normalsize1}}
\newcommand{\op}{{\mathrm{op}}}
\newcommand{\sop}{{\mathrm{sop}}}
\newcommand{\Ab}{\mathbf{Ab}}
\newcommand{\xx}{\mathbf{x}}

\newcommand{\Set}{\mathbf{Set}}
\newcommand{\vecc}{\mathbf{Vec}}
\newcommand{\Hpf}{\mathbf{Hpf}}

\newcommand{\Spec}{{\rm Spec} }

\newcommand{\minus}{\scalebox{0.9}{{\rm -}}}

\newcommand{\Alg}{\mathbf{Alg}}
\newcommand{\Grp}{\mathbf{Grp}}

\newcommand{\gid}{\mathsf{e}}

\newcommand{\bpi}{\mathbf{\Pi}}

\newcommand{\mH}{\mathsf{H}}
\newcommand{\mR}{\mathbb{R}}
\newcommand{\mC}{\mathbb{C}}

\newcommand{\mZ}{\mathbb{Z}}

\newcommand{\FF}{\mathbf{F}}

\newcommand{\cD}{\mathcal{D}}

\newcommand{\cC}{\mathcal{C}}

\newcommand{\cO}{\mathcal{O}}

\newcommand{\cA}{\mathcal{A}}

\newcommand{\cB}{\mathcal{B}}
\newcommand{\cZ}{\mathcal{Z}}

\newcommand{\cX}{\mathcal{X}}

\newcommand{\oa}{\bar{0}}
\newcommand{\ob}{\bar{1}}

\newcommand{\End}{{\rm End}}

\newcommand{\Aut}{{\rm Aut}}

\newcommand{\Id}{{\rm Id}}

\newcommand{\im}{{\rm Im}}

\newcommand{\charr}{{\rm char}}

\newcommand{\mG}{\mathbb{G}}
\newcommand{\mT}{\mathbb{T}}

\newcommand{\mk}{\Bbbk}

\begin{document}

\title[The $G$-centre]{The $G$-centre and gradable derived equivalences}

\author{Kevin Coulembier and Volodymyr Mazorchuk}
\date{}

\begin{abstract}
We propose a generalisation for the notion of the centre of an algebra
in the setup of algebras graded by an arbitrary abelian group $G$.
 Our generalisation, which we call the
$G$-centre, is designed to control the endomorphism category of the
grading shift functors. We show that the $G$-centre is preserved
by gradable derived equivalences given by tilting modules. We also discuss 
links with existing notions in superalgebra theory and apply our results 
to derived equivalences of superalgebras.
\end{abstract}

\maketitle


\vspace{5mm}

\section{Introduction}
Consider a finite dimensional algebra $A$ over a field $\mk$ and the corresponding category
$A$-mod of finite dimensional left $A$-modules. In this setup, the evaluation of a natural 
endomorphism of the identity functor $\Id$ on $A$-mod at the left regular $A$-module ${}_AA$ 
gives rise to the classical isomorphism
\begin{equation}\label{CentreId}
\cZ(A)\;\cong\; \End(\Id),
\end{equation}  
between the centre of an algebra and the centre of its module category. In \cite[Proposition~9.2]{Derived1}, 
Rickard proved that two derived equivalent algebras have isomorphic centres, providing a fundamental 
invariant for the study of derived equivalences.
When the algebras in question are graded by some group and the derived equivalence suitably 
preserves this grading, it is easy to show that the centres are isomorphic even as graded algebras. 
In this paper we take a slightly different view at this situation and introduce a new larger algebra
that  extends the classical centre of an algebra which we show is 
preserved by so-called `gradable derived equivalences' between graded algebras which are given by tilting modules.

A motivating example is given by the theory of superalgebras. When associative $\mZ_2$-graded algebras 
are interpreted as `superalgebras', there is an alternative notion of the centre, known as {\em super centre}. 
Furthermore, in \cite{Gorelik}, Gorelik introduced the notion of the {\em ghost centre} of a superalgebra.
This ghost centre  is a certain subalgebra containing both the centre and the super centre which turned out to play
a very important role in studying representations of Lie superalgebras. The natural questions which originated
the present study are whether the super centre and the ghost centre could be realised as natural transformations 
for some endofunctors on the module category and whether 
these subalgebras are preserved under (certain) derived equivalences.

We start our investigation in a different setting, namely, that of an algebra $A$ on which an arbitrary
group $H$ acts by automorphisms. This allows us to define the {\em extended centre}, which is not a 
subalgebra of~$A$, but, rather, a subalgebra of  $A\otimes\mk H$. The group action on $A$ leads to a strict categorical 
action of~$H$ on the (derived) module category of~$A$. We show that the extended centre can be 
realised as the algebra of natural transformation of the functors which yield this strict categorical action. 
Furthermore, we prove that certain derived equivalences which intertwine the actions in a suitable way 
preserve the extended centres of involved algebras.

If the algebra $A$ is graded by an abelian group $G$, the grading can be reformulated in terms of 
an action of the character group $\hat{G}=\Hom(G,\mk^\times)$, with respect to the ground field $\mk$.
When $|G|$ is finite and not divisible by $\charr(\mk)$, the notions of $\hat{G}$-actions and $G$-gradings
are actually equivalent.

For a grading on $A$ by an arbitrary abelian group $G$, we introduce the $G$-centre, which is a subalgebra
of the algebra of functions from $G$ to~$A$. When $|G|$ is finite and not divisible by $\charr(\mk)$, we show
that the $G$-centre is isomorphic to the extended centre, corresponding to the $\hat{G}$-action. In general
the two notions differ. We show how the $G$-centre can be realised as the algebra 
of natural transformations of certain functors on the category of graded modules. Then we prove that 
the $G$-centre is preserved under  `gradable derived equivalences', as introduced in~\cite{CM4}, 
provided that the equivalence is given in terms of a tilting module.

While our current methods do not allow to consider derived equivalences in full generality, we hope that 
the condition that the derived equivalence be given by a tilting module can be lifted using a different
approach. On the other hand, the results in \cite{CM4} show for example that, for any two blocks of category $\cO$
in type A which are gradable derived equivalent (for the Koszul $\mZ$-grading), one can construct a 
gradable derived equivalence between them which is given by a tilting module.

Then, we return to the special case of~$G=\mZ_2$, thus of that of superalgebras. Our notion 
of~$G$-centre is very closely related to the ghost centre. Concretely, it is isomorphic to an 
exterior direct sum of the super centre and the anti centre, whereas the ghost centre is the 
sum (not necessarily direct) of the super centre and the anti centre inside the algebra~$A$. 
The two notions are thus only different in case some non-zero elements of 
$A$ belong to the super and anti centre at the same time, so we can view the
$G$-centre as a natural lift of the ghost centre. Our general 
results then yield concrete methods to realise the super centre (and the $G$-centre) as 
endomorphism algebras of certain functors on the supermodule category of a superalgebra.
Furthermore, our results show that the super centre and the $G$-centre are both preserved 
under the most canonical definition of derived equivalences between superalgebras. This 
provides an answer to both our original motivating questions.

The paper is organised as follows. In Section \ref{NotConv} we fix some notation and conventions. In Section~\ref{GroupAc} we study actions of finite groups on algebras, modules and categories. 
In Section~\ref{SecExtCen}, we obtain our results on the extended centre.
In Section~\ref{Gradings} we establish some elementary properties of~$G$-gradings. 
In Section~\ref{SecGcentre}, we obtain our results on the $G$-centre. In Section~\ref{SecSuper}, 
we apply our results to superalgebras and compare with some existing notions in the literature. 
In Section~\ref{SecHoch} we point out some natural questions for future research, related to 
Hochschild cohomology. In Appendix~\ref{AppGroup}, we give details on two technical proofs of  
statements in Section~\ref{GroupAc} related to strict categorical group actions. In Appendix~\ref{AppB}
we show that some properties of tilting modules that we apply will fail when considering
general tilting complexes.

\section{Notation and conventions}\label{NotConv}

We fix an algebraically closed field $\mk$.  We denote by $\Set$ the category of sets and by $\Ab$ the category of abelian groups. The category of~$\mk$-vector spaces is denoted by $\vecc_{\mk}$. The category of associative unital $\mk$-algebras is denoted by $\Alg$. By `algebra' we will mean an object in $\Alg$. All unspecified categories and functors are 
assumed to be $\mk$-linear and additive. The category of~$\mk$-linear additive functors on 
a $\mk$-linear additive category $\cC$ is denoted by $\Fun(\cC)$.

Consider categories $\cA,\cB,\cC$ and~$\cD$; functors $F: \cA\to\cB$,~$H:\cC\to\cD$ and 
functors $G_1,G_2:\cB\to\cC$ with a natural transformation~$\eta:G_1\Rightarrow G_2$. We will 
use the natural transformation~$H(\eta): H\circ G_1\Rightarrow H\circ G_2$, where $H(\eta)_X:=H(\eta_X)$, 
for any object $X$ in $\cB$. The natural transformation~$\eta_F: G_1\circ F\Rightarrow G_2\circ F$ 
is given by $(\eta_F) _Y:= \eta_{F(Y)}$, for any object $Y$ in $\cA$.
For an exact functor $F$ between two abelian categories $\cA$ and~$\cB$, we will use the 
 notation~$F_\bullet$ for the corresponding triangulated functor $\cD^b(\cA)\to \cD^b(\cB)$ 
acting between the corresponding  bounded derived categories. 

The multiplicative 
identity of an algebra $A\in\Alg$ will be denoted by~$1_A$, or $1$ if there is no confusion possible. 
We denote the group of $\mk$-algebra automorphisms of~$A$ by $\Aut(A)$. If the algebra $A$ is finite dimensional, we denote by $A$-mod the category of 
finite dimensional left $A$-modules. 

We will abbreviate $\cD^b(\mbox{{\rm A-mod}})$ to~$\cD^b(A)$. We will say that a triangulated equivalence $F:\cD^b(A)\stackrel{\sim}{\to}\cD^b(B)$ is {\em strong} if both $F({}_AA)$ and $F^{-1}({}_BB)$ are quasi-isomorphic to complexes contained in one degree. The corresponding modules are then {\em tilting modules}, see Appendix~\ref{AppB}.

For an arbitrary group $H$, we denote its identity element by $\gid=\gid_H$. The category of $\mk$-linear representations of $H$ will be denoted by $\Rep_{\mk}H$. Its objects are thus pairs $(V,\psi)$, with $V\in \vecc_{\mk}$ and $\psi$ a group homomorphism 
$$\psi:H\to\Aut_{\mk}(V),\quad h\mapsto \psi_h.$$
In this way, we have $\psi_h\circ\psi_k=\psi_{hk}$, for arbitrary $h,k\in H$, and $\psi_{\gid}=1_V$.

We denote the group (Hopf)
algebra of $H$ by $\mk H$.  For $A\in\Alg$, we consider $\Hom_{\mk}(\mk H,A)=\Hom_{\Set}(H,A)$ as an algebra with pointwise multiplication. In particular, we write $\mk^H=\Hom_{\Set}(H,\mk)$.





\section{Group actions}\label{GroupAc}

In this section we introduce some notions related to strict categorical actions of groups. 
Technical proofs of Propositions~\ref{InterInv} and~\ref{PropCDEq} are given in Appendix~\ref{AppGroup}. We fix a group $H$. 

\subsection{Group actions on algebras and modules}

\subsubsection{Compatible actions}\label{ActAlg}
An action of~$H$ on an 
algebra~$A$ is 
defined to be a group homomorphism $\phi:H\to \Aut(A)$,~$f\mapsto \phi_f$. In other words, $(A,\phi)\in \Rep_{\mk}H$ and the image of $\phi$ consists of algebra automorphisms. We can and will identify 
$H$-actions on~$A$ and~$A^{\op}$.
 Although not essential for this paper, we note that an action of~$H$ on $A$ as defined above is
equivalent to the notion of a Hopf $\mk H$-module algebra structure on $A$.

Assume we have $(V,\psi)\in \Rep_{\mk}H$ such that $V$ is, additionally, an $A$-module. The actions of 
$H$ on~$A$ and~$V$ are said to be {\em compatible}
 if $\psi_h(av)=\phi_h(a)\psi_h(v)$, for all $h\in H$,~$a\in A$ and~$v\in V$.

For any $\alpha\in \Aut(A)$ and any $A$-module $M$ with underlying vector space $V$, we 
denote by ${}^\alpha M$ the $A$-module with underlying vector space $V$, but with the 
action of~$a\in A$ on~$v\in V$ given by $\alpha(a)\cdot v$. The above notion of 
compatibility is thus equivalent to~$\psi_h\in\Hom_A(M, {}^{\phi_h}M)$.




\subsubsection{The Hopf smash products}\label{AphiH}
For a group action~$\phi:H\to\Aut(A)$, we have the Hopf smash product $A\#  \mk H=A\# H$. As a vector space, this is $A\otimes \mk H$ with multiplication
$$(a,h)(b,k)=(a\phi_h(b),hk).$$
We will also use $A^{\op}\# H$, which has multiplication
$$(a,h)(b,k)=(\phi_h(b)a,hk).$$

\subsection{Group actions on categories}

\subsubsection{Strict categorical actions}\label{StrictSec}

Let~$\Gamma$ be a strict categorical action of $H$ on a category~$\cC$, 
{\it i.e.} we have $\mk$-linear endofunctors $\Gamma_h$ on~$\cC$, for each $h\in H$, with $\Gamma_{\gid}=\Id$
and~$\Gamma_{h_1}\circ \Gamma_{h_2}\,=\, \Gamma_{h_1h_2}$. 

For any object $X$ in $\cC$, we introduce the $\mk$-vector space
$$\End(\Gamma;X)\;:=\;\bigoplus_{h\in H}\Hom_{\cC}(X,\Gamma_hX).$$ 
This space has the structure of an algebra given by
$$\Hom_{\cC}(X,\Gamma_hX)\otimes \Hom_{\cC}(X,\Gamma_kX)\to \Hom_{\cC}(X,\Gamma_{kh}X),\quad \alpha\otimes\beta\mapsto \Gamma_k(\alpha)\circ\beta.$$


In a similar fashion, we can consider the algebra 
$$\End(\Gamma)\;:=\;\bigoplus_{h\in H}\Nat(\Id,\Gamma_h).$$
The following statement follows directly from the definitions.

\begin{lemma}\label{LemEva}
For any object $X$ in $\cC$, evaluation yields an algebra morphism
$$\Ev^\Gamma_X:\;\End(\Gamma)\;\to\; \End(\Gamma; X);\qquad \eta\mapsto \eta_X.$$
\end{lemma}

In some cases we will need a more refined evaluation.

\begin{definition}\label{DefEvPi}
{\rm
The {\em astute evaluation} is an algebra morphism,
$$\Delta\Ev^{\Gamma}_{X}:\; \,\End(\Gamma)\;\to\;  
\Hom_{\Set}\left(H, \End(\Gamma; X)\right),$$
which is given by
$$\Nat(\Id,\Gamma_h)\ni \eta\;\mapsto\; \{g\mapsto \Gamma_{g^{\minus 1}}(\eta_{\Gamma_gX})\,|\,g\in H\}.$$
}
\end{definition}

\subsubsection{Intertwining categorical group actions}\label{CatComm}

Let~$\Gamma$, resp. $\Upsilon$, be strict categorical actions of $H$ on a category~$\cC$, 
resp. $\cD$. We say that a $\mk$-linear {\em functor $K:\cC\to\cD$ 
intertwines the actions $\Gamma$ and~$\Upsilon$} if we have natural transformations
$$\xi^h: K\circ \Gamma_h\Rightarrow \Upsilon_h\circ K,\qquad\mbox{for all}\quad h\in H,$$
where $\xi^{\gid}=\Id_{K}$ and the relation
\begin{equation}\label{eqCommFun}\Upsilon_k(\xi^h)\,\circ\,\xi^k_{\Gamma_h} \;=\; \xi^{kh}
\end{equation}
is satisfied, for all $h,k\in H$. The condition in Equation~\eqref{eqCommFun} is equivalent to saying that the diagram
\begin{equation}\label{eqCommFun2}
\xymatrix{
K\circ \Gamma_{kh}=K\circ \Gamma_k\circ \Gamma_h\ar[rr]^{\xi^k_{\Gamma_h}}\ar[drr]^{\xi^{kh}}&& 
\Upsilon_k\circ K\circ \Gamma_h\ar[d]^{\Upsilon_k(\xi^h)}\\
&   &\Upsilon_{kh}\circ K=\Upsilon_k\circ \Upsilon_h\circ K}
\end{equation}
commutes, for all $h,k\in H$.
The above conditions imply, in particular, that, for any object $X$ in $\cC$ and any $h\in H$, 
the morphism $\xi^h_{\Gamma_{h^{\minus 1}} X}$ is invertible, with inverse $\Upsilon_h(\xi^{h^{\minus 1}})$. 
As the functor $\Gamma_{h^{\minus 1}}$ has inverse $\Gamma_h$, this implies that the natural 
transformation~$\xi^h$ is an isomorphism of functors. 

In the particular case where one has the equality $K\circ \Gamma_h=\Upsilon_h\circ K$, 
for all $h\in H$, 
we can take all $\xi^h$ to be the identity natural transformations and the condition in 
equation~\eqref{eqCommFun} is automatically satisfied.

\begin{proposition}\label{InterInv}
Assume that the functor $K$ which intertwines the actions $\Gamma$ and~$\Upsilon$ as above, 
has a (weak) inverse $K^{\minus 1}$ given by isomorphisms $\alpha : K^{\minus 1}\circ K\Rightarrow \Id$ 
and~$\beta:\Id \Rightarrow K\circ K^{\minus 1}$ 
making  $(K,K^{\minus 1})$ a pair of adjoint functors. Then we introduce the natural transformations
$$\eta^h:\;\;K^{\minus 1}\circ\Upsilon_h\Rightarrow \Gamma_h\circ K^{\minus 1}$$
defined as
$$\eta^h=\alpha_{\Gamma_h\circ K^{\minus 1}}\circ K^{\minus 1}((\xi^h)^{\minus 1})_{K^{\minus 1}}
\circ (K^{\minus 1}\circ\Upsilon_h)(\beta).$$
This corresponds to the composition
$$K^{\minus 1}\circ\Upsilon_h\Rightarrow K^{\minus 1}\circ\Upsilon_h \circ K\circ K^{\minus 1}\Rightarrow 
K^{\minus 1} \circ K\circ\Gamma_h\circ K^{\minus 1}\Rightarrow \Gamma_h\circ K^{\minus 1}.$$
With this definition, the $\{\eta^h\}$ satisfy the intertwining relations \eqref{eqCommFun} for~$K^{\minus 1}$.
\end{proposition}

For the proof of Proposition~\ref{InterInv}, see Appendix~\ref{AppGroup}.

When $\cC=\cD$ and~$\Gamma=\Upsilon$, we simply say that $K$ {\em commutes with the categorical $H$-action~$\Gamma$}.

\subsubsection{Categorical actions and equivalences}
Consider an equivalence $F:\cC\;\tilde\to\;\cD$ of categories. This induces an equivalence of categories
$$\FF\;:\;\Fun(\cC)\;\tilde\to\;\Fun(\cD),$$
where $\FF(K):=F\circ K\circ F^{\minus 1}$, for a functor $K$ (an object in $\Fun(\cC)$), and~$\FF(\eta)=F(\eta)_{F^{\minus 1}}$, for a natural transformation~$\eta$ (a morphism in $\Fun(\cC)$). We point out that the equivalence $\FF$ does not necessarily respect composition of functors
(it only does it up to isomorphism). In particular, one cannot expect $\FF$ to map a set of functors forming a {\em strict} group action to a set of functors with the same property. In the following we will continue to refer to objects in $\Fun(\cC)$ simply as `functors' and morphisms in $\Fun(\cC)$ as natural transformations.

\begin{proposition}\label{PropCDEq}
Consider an equivalence $F:\cC\;\tilde\to\;\cD$ which intertwines strict $H$-actions $\Gamma$ on~$\cC$ and~$\Upsilon$ on~$\cD$. Then there is an algebra isomorphism
$$\End(\Gamma)\;\stackrel{\sim}{\to}\;\End(\Upsilon),$$
\end{proposition}


For the proof of Proposition~\ref{PropCDEq}, see Appendix~\ref{AppGroup}.


Naturally, the analogue of Proposition~\ref{PropCDEq} for evaluations of functors is also true.
\begin{lemma}\label{LemInt2}
With assumptions as in Proposition~\ref{PropCDEq} and for an object $X\in\cC$, we have an algebra isomorphism
$$\End(\Gamma; X)\;\cong\;\End(\Upsilon; FX).$$
\end{lemma}

\subsubsection{Category of modules}\label{phiPhi}
A group action~$\phi$ on the algebra $A$ induces a group action~$\Phi$ on the category $A$-mod as follows. 
For any $h\in H$, let $\Phi_h$ denote the functor on~$A$-mod, which preserves the underlying vector space of modules and preserves morphisms between modules, but twists the $A$-action by $\phi_{h^{\minus 1}}=\phi^{\minus 1}_h$. This leads to a categorical group action indeed, as, for any $M\in A$-mod, we have
$$\Phi_h\circ\Phi_g(M)={}^{\phi_{h^{\minus 1}}}\left({}^{\phi_{g^{\minus 1}}}M\right)= {}^{\phi_{g^{\minus 1}}\circ \phi_{h^{\minus 1}}}M = \Phi_{hg}(M).$$

\subsubsection{Actions on objects in categories}Consider a category $\cC$ with a strict action~$\Phi$ of~$H$ and an object $X$ in $\cC$. We will now formalise the concept of a compatible action on a module of~\ref{ActAlg} and use this to define an action on endomorphism algebras.

\begin{definition}\label{defXAlg2}
{\rm 
A set of morphisms $\psi=\{\psi_h,\,h\in H\}$, with
$$\psi_h\in\Hom_{\cC}(X,\Phi_{h^{\minus 1}}X)\qquad\mbox{and}\qquad
\Phi_{h^{\minus 1}}(\psi_k)\circ\psi_h\;=\;\psi_{kh},$$
and $\psi_\gid=1_X$, is called a $\Phi$-compatible 
$H$-action on the object $X$.
If $X$ admits a $\Phi$-compatible $H$-action~$\psi$, the algebra $\End_{\cC}(X)$ 
admits an $H$-action~$\theta=\theta_X^{(\Phi,\psi)}$ given by
$$\theta_{g}(\alpha)\;=\;\Phi_{g}(\psi_g\circ\alpha)\circ\psi_{g^{\minus 1}},$$
for all $g\in H$ and~$\alpha\in \End_{\cC}(X)$.
}
\end{definition}

One checks, by direct computation, that the above action is well-defined, meaning $\theta_h\circ\theta_g(\alpha)=\theta_{hg}(\alpha)$ and~$\theta_g(\alpha\circ\beta)=\theta_g(\alpha)\circ\theta_g(\beta)$. 

\begin{example}\label{ExTheta}{\rm
Take $\cC=A$-mod and~$\Phi$ induced from an $H$-action~$\phi:H\to\Aut(A)$ as in~\ref{phiPhi}. 
We can interpret $\phi_h$ as an element of~$\Hom_A(A,{}^{\phi_h}A)$, for each $h\in H$. 
The relation~$\Phi_{h^{\minus 1}}(\phi_k)\circ\phi_h=\phi_{kh}$ follows immediately from 
the interpretation of both morphisms in $\End_{\mk}(A)$. Hence, Definition~\ref{defXAlg2} 
allows us to introduce an $H$-action~$\theta=\theta^{\Phi,\phi}_{A}$ on~$\End_A(A)\cong A^{\op}$. 
It follows from direct computation that this can be identified with the original $H$-action~$\phi$.
}\end{example}

\begin{lemma}\label{LemAlgX}
Under the assumptions of Definition~\ref{defXAlg2}, we have an algebra isomorphism
$$\End(\Phi; X)\;\;\tilde\to\;\; \End_\cC(X)\# H,$$
where $\alpha\in \Hom_{\cC}(X,\Phi_{h^{\minus 1}}X)$ is mapped to~$(\Phi_h( \alpha)\circ\psi_{h^{\minus 1}},h)$. 
\end{lemma}
\begin{proof}
We have mutually inverse morphisms of vector spaces given by
$$\Hom_{\cC}(X,\Phi_{h^{\minus 1}}X)\to \End_{\cC}(X);\quad \alpha\mapsto \Phi_h( \alpha)\circ\psi_{h^{\minus 1}},$$
and 
$$\End_{\cC}(X)\to\Hom_{\cC}(X,\Phi_{h^{\minus 1}}X);\quad \alpha\mapsto \Phi_{h^{\minus 1}}( \alpha)\circ\psi_{h}.$$

Hence,
the proposed morphism is an isomorphism of vector spaces. For 
any elements $\alpha\in \Hom_{\cC}(X,\Phi_{h^{\minus 1}}X)$ and~$\beta\in \Hom_{\cC}(X,\Phi_{k^{\minus 1}}X)$, we have
$\alpha\beta=\Phi_{k^{\minus 1}}(\alpha)\circ\beta$, which is mapped to 
$$(\Phi_h(\alpha)\circ \Phi_{hk}(\beta)\circ \psi_{(hk)^{\minus 1}},hk).$$
On the other hand, by~\ref{AphiH} and Definition~\ref{defXAlg2}, the product of~$(\Phi_h( \alpha)\circ\psi_{h^{\minus 1}},h)$ and~$(\Phi_k( \beta)\circ\psi_{k^{\minus 1}},k)$ inside $\End_{\cC}(X)\#H$ is given by
$$(\Phi_h( \alpha)\circ\psi_{h^{\minus 1}}\circ\theta_h(\Phi_k( \beta)\circ\psi_{k^{\minus 1}}),hk)=(\Phi_h( \alpha)\circ\psi_{h^{\minus 1}}\circ\Phi_h(\psi_h)\circ\Phi_{hk}(\beta)\circ \Phi_h(\psi_{k^{\minus 1}})\circ\psi_{h^{\minus 1}},hk),$$
and the claim follows.
\end{proof}


\section{Extended centre}\label{SecExtCen}
We fix a group $H$ and a finite dimensional algebra $A$, for which there is a group homomorphism $\phi:H\to \Aut(A)$,~$f\mapsto \phi_f$. 

\begin{definition}\label{DefExtCen}
{\rm
The $\phi$-extended centre $\cZ^\phi(A)$ of~$A$ is the subalgebra of~$A\otimes\mk H$, spanned by all
$(a,f)$, where $a\in A$ and~$f\in H$, such that
$$a\,b\;=\; \phi_f(b)\,a,\qquad\mbox{for all }\, b\in A. $$
}
\end{definition}

The fact that $\cZ^{\phi}(A)$ is closed under multiplication on~$A\otimes\mk H$ is immediate. 
Recalling the definition of the algebras in~\ref{AphiH} leads to the following lemma.

\begin{lemma}\label{LemZeta}
{\hspace{2mm}}

\begin{enumerate}[$($i$)$]
\item\label{LemZeta.1} The subalgebra $\zeta^\phi(A)$ of~$A^{\op}\# H $ given by elements $(a,h)$ satisfying
$$(a,h)(b,k)=(ab,hk),\qquad\mbox{ for all }\;(b,k)\in A^{\op}\# H ,$$
is isomorphic to~$\cZ^\phi(A)$.
\item\label{LemZeta.2}
The subalgebra $\zeta_\phi(A)$ of~$A\# H$ given by elements $(a,h)$ satisfying
$$(a,h)(b,k)=(ba,hk),\qquad\mbox{ for all }\;(b,k)\in A\# H,$$
is isomorphic to~$\cZ^\phi(A^{\op})$.
\end{enumerate}
\end{lemma}

\subsection{Categorical formulation}

We use the notions introduced in~\ref{StrictSec} for the categorical group action~$\Phi$ on~$A$-mod obtained from $\phi$ as in~\ref{phiPhi}. The main result of this subsection is the following theorem, which is a generalisation of Equation~\eqref{CentreId}.

\begin{theorem}\label{ThmZphiFunctor}
We have an algebra isomorphism
$$\cZ^{\phi}(A)\;\cong\;\End(\Phi),$$
under which $(a,h)\in \cZ^{\phi}(A)$ is identified with $\eta:\Id\Rightarrow\Phi_{h^{-1}}$, where $\eta_M:M\to{}^{\phi_h}M$ is given by $\eta_M(v)=av$, for any $A$-module $M$ and all $v\in M$. 
\end{theorem}

\begin{remark}{\rm
The combination of Theorem~\ref{ThmZphiFunctor} and Proposition~\ref{InterInv} implies an isomorphism between the
extended centres  of two Morita equivalent algebras with $H$-actions for which the induced 
$H$-actions on their module  categories are intertwined by the Morita equivalence. We will generalise this 
statement in Theorem~\ref{ThmDerEqPhi}.
}\end{remark}

Now we start the proof of Theorem~\ref{ThmZphiFunctor}.

\begin{lemma}\label{LemApH}
There is an algebra isomorphism
$$\End(\Phi; A)=\bigoplus_{h\in H}\Hom_A(A,{}^{\phi_h}A)\,\to\, A^{\op}\# H,$$
which maps $\alpha\in\Hom_A(A,{}^{\phi_{h}}A)$ to~$(\alpha(1),h)$.\end{lemma}

\begin{proof}
The proposed morphism is, clearly, an isomorphism of vector spaces. Now, consider $\alpha: A\to {}^{\phi_{h}}A$ with $a:=\alpha(1)$ and~$\beta: A\to {}^{\phi_{k}}A$ with $b:=\beta(1)$. Then $\alpha\beta=\Phi_{k^{\minus 1}}(\alpha)\circ\beta :A\to {}^{\phi_{hk}}A$, so we have $\alpha\beta(1)=\phi_h(b)a$. Hence $\alpha\beta$ gets mapped to~$(\phi_{h}(b)a,hk)$, meaning that we obtain indeed an algebra isomorphism.
\end{proof}

\begin{lemma}\label{ExistNatT}
For each element $(a,h)\in \cZ^\phi(A)$, there exists a natural transformation
$\eta:\Id\Rightarrow\Phi_{h^{\minus 1}}$ such that $\eta_M:M\to{}^{\phi_h}M$ is given 
by $\eta_M(v)=av$, for any $A$-module $M$ and all $v\in M$.
\end{lemma}
\begin{proof}
That $\eta_M$ is $A$-linear follows from the definition of~$\cZ^{\phi}(A)$. For a morphism $\alpha:M\to N$, we have $\eta_N\circ \alpha=\Phi_{h^{-1}}(\alpha)\circ \eta_M$, which follows immediately from the fact that $\Phi_{h^{-1}}(\alpha)=\alpha$ as morphisms of~$\mk$-vector spaces. Thus the family $\{\eta_M\}$ yields indeed a natural transformation.
\end{proof}

Now we study the evaluation in Lemma \ref{LemEva} for the left regular $A$-module. Evaluation is then automatically injective since $A$ is a projective generator.

\begin{lemma}\label{LemImEv}
Denote the composition of 
the map $\Ev^{\Phi}_A:\End(\Phi)\hookrightarrow \End(\Phi; A)$ 
with the isomorphism in Lemma~\ref{LemApH} by
$$\overline{\Ev}^{\Phi}_A: \End(\Phi)\hookrightarrow A^{\op}\# H .$$ 
Then the image of~$\overline{\Ev}^{\Phi}_A$ coincides with the subalgebra 
$\zeta^\phi(A)\subset A^{\op}\# H$ in Lemma~\ref{LemZeta}\eqref{LemZeta.1}. 
\end{lemma}

\begin{proof}
Consider a natural transformation~$\eta: \Id\Rightarrow \Phi_{h^{\minus 1}}$. Evaluation of~$\eta$ yields a morphism $\eta_A:A\to{}^{\phi_h}A$, which fits into a commutative diagram
\begin{displaymath}
    \xymatrix{
       A\ar[rr]^{\eta_A}\ar[d]_{\beta}&& \Phi_{h^{\minus 1}}A\ar[d]^{\Phi_{h^{\minus 1}}(\beta)}\\
      A\ar[rr]^{\eta_A}&   & \Phi_{h^{\minus 1}}A,}
\end{displaymath}
for any morphism $\beta\in \End_A(A)\cong A^{\op}$. We take an arbitrary $b\in A$ and the corresponding $\beta\in \End_A(A)$ such that $\beta(1)=b$. The condition that the above diagram commutes is then equivalent to the equality $\eta_A(1) b=\phi_h(b)\eta_A(1)$. We set $a:=\eta_A(1)\in A$ and thus find that $\im(\overline{\Ev}^{\Phi}_A)$ corresponds to those $(a,h)\in A^{\op}\# H $ for which we have $ab=\phi_h(b)a$, for all $b\in A$. The definition of~$A^{\op}\# H $ in \ref{AphiH} implies that we can characterise these elements $(a,h)$ equivalently by the condition
$$(a,h)(b,k)=(ab,hk),$$
for all $(b,k)\in A^{\op}\# H$. 
\end{proof}

\begin{proof}[Proof of Theorem~\ref{ThmZphiFunctor}]
The proposed isomorphism is induced by Lemma~\ref{LemZeta}\eqref{LemZeta.1} and $\overline{\Ev}^{\Phi}_A$ in Lemma~\ref{LemImEv}. The stated properties of the isomorphism follow by definition of~$\overline{\Ev}^{\Phi}_A$.
\end{proof}

\subsection{Derived equivalences}
The main result of this subsection is the following theorem, which can be viewed as a generalisation of
\cite[Proposition~9.2]{Derived1}. We denote by $\cD^b(A)$ the bounded derived category of the abelian category $A$-mod.

\begin{theorem}\label{ThmDerEqPhi}
Let $A,B$ be finite dimensional algebras equipped with $H$-actions $\phi:H\to\Aut(A)$ and~$\omega:H\to\Aut(B)$, respectively. Let
$$F\,:\;\cD^b(A)\;\,\tilde\to\,\;\cD^b(B)$$
be an equivalence of triangulated categories such that $F$ intertwines $\Phi$ and~$\Omega$ (the categorical actions on $\cD^b(A)$ and $\cD^b(B)$ corresponding to~$\phi$ and~$\omega$) as in \ref{CatComm}. Then $F$ induces an algebra morphism
$$\cZ^{\phi}(A)\;\to\;\cZ^{\omega}(B),$$
which is an isomorphism if $F$ is a strong derived equivalence.
\end{theorem}

\begin{proof}
Let $\xi^h:F\circ \Phi_h\Rightarrow \Omega_h\circ F$ be the natural transformations which give the intertwining relations.
Let $T_\bullet\in\cD^b(B)$ be the complex $F({}_AA)$. For each $h\in H$, we define
$$\psi_h\,:=\, \xi_A^{h^{\minus 1}}\circ F(\phi_h)\;\,\in\;\, \Hom_{\cD^b(B)}(T_\bullet,\Omega_{h^{\minus 1}}T_\bullet),$$
where we interpret $\phi_h$ as an element of~$\Hom_A(A,\Phi_{h^{\minus 1}}A)$.
We calculate, using the definition of~$\xi^{k^{\minus 1}}$ and Equation~\eqref{eqCommFun},
\begin{eqnarray*}
\Omega_{k^{\minus 1}}(\psi_h)\circ\psi_k &=& \Omega_{k^{\minus 1}}(\xi_A^{h^{\minus 1}})\circ (\Omega_{k^{\minus 1}}\circ F)(\phi_h)\circ \xi_A^{k^{\minus 1}}\circ F(\phi_k)\\
&=& \Omega_{k^{\minus 1}}(\xi_A^{h^{\minus 1}})\circ \xi_{{}^{\phi_h}A}^{k^{\minus 1}}\circ (F\circ\Phi_{k^{\minus 1}})(\phi_h)\circ F(\phi_k)\\
&=& \xi_{A}^{(hk)^{\minus 1}}\circ F((\Phi_{k^{\minus 1}})(\phi_h)\circ \phi_k)=  \xi_{A}^{(hk)^{\minus 1}}\circ F(\phi_{hk})=\psi_{hk}.
\end{eqnarray*}
Hence, $\psi$ yields an $\Omega$-compatible $H$-action on~$T_\bullet$ and we can apply Definition~\ref{defXAlg2} to define an action~$\theta=\theta^{\Omega,\psi}_{T_\bullet}:H\to \Lambda:=\End_{\cD^b(B)}(T_\bullet)$.
We claim that, under the algebra isomorphism $A^{\op}\to \Lambda$ induced by $A^{\op}\cong \End_A(A)$ and~$F$, the action~$\theta$ corresponds to the action~$\phi$. To prove this, we consider $\alpha\in \End_A(A)$ and calculate
\begin{eqnarray*}
\theta_{h}(F(\alpha))&=&(\psi_{h^{\minus 1}})^{\minus 1}\circ (\Omega_{h}\circ F)(\alpha)\circ \psi_{h^{\minus 1}}\\
&=&(\xi_A^{h}\circ F(\phi_{h^{\minus 1}}))^{\minus 1}\circ (\Omega_{h}\circ F)(\alpha)\circ \xi_A^{h}\circ F(\phi_{h^{\minus 1}})\\
&=&(\xi_A^{h}\circ F(\phi_{h^{\minus 1}}))^{\minus 1}\circ  \xi_A^{h}\circ (F\circ\Phi_{h})(\alpha)\circ F(\phi_{h^{\minus 1}})\\
&=& F( (\phi_{h^{\minus 1}})^{\minus 1} \circ \Phi_{h}(\alpha)\circ \phi_{h^{\minus 1}}).
\end{eqnarray*}
The claim then indeed follows from Example~\ref{ExTheta}. This means, in particular, that $\Lambda\#H\;\cong\; A^{\op}\#H$.

Combining this with Lemma~\ref{imcom} in Subsection~\ref{Evaluation} and Lemma~\ref{LemZeta}\eqref{LemZeta.2}, yields an algebra morphism
$$\cZ^\omega(B)\;\stackrel{\zeta_{T_\bullet}}{\to}\; \zeta_{\theta}(\Lambda)\;\stackrel{\sim}{\to}\; \zeta_\phi(A^{\op})\;\stackrel{\sim}{\to}\;  \cZ^\phi(A).$$
If $F$ is a strong equivalence, then $T_\bullet$ is a tilting module
and Lemma~\ref{PropSumm} implies that this composition is injective. 
The corresponding reasoning for~$F^{\minus 1}$, using Proposition~\ref{InterInv}, 
gives an inclusion in the other direction. Note that $\cZ^\phi(A)$ is a subalgebra of $A\otimes\Bbbk  H$,
$\cZ^\omega(B)$ is a subalgebra of $B\otimes\Bbbk H$ and the above maps respect $H$ in the sense that
they map an element of the form $(a,f)$ to an element of the form $(b,f)$. As both $A$ and $B$ are
finite dimensional, bijectivity of both maps above follows from their injectivity. This completes the proof.
\end{proof}

\subsection{Evaluation}\label{Evaluation}
In this subsection we let $X_\bullet$ be an arbitrary object in $\cD^b(A)$ which admits a $\Phi$-compatible $H$-action~$\psi$.
This means that we can apply Definition~\ref{defXAlg2} to construct an $H$-action~$\theta$ on~$\End_{\cD^b(A)}(X_\bullet)$.

\begin{definition}\label{defzeta}
{\rm 
With $\Lambda:=\End_{\cD^b(A)}(X_\bullet)$, we let
$$\zeta_{X_\bullet}\;:\;\cZ^\phi(A)\,\to\, \Lambda\#H
$$
denote the composition
 $$\cZ^\phi(A)\;\;\tilde\to\;\; \End(\Phi)\;\hookrightarrow\;\End(\Phi_\bullet)\;\stackrel{\Ev^{\Phi}_{X_\bullet}}{\to}\; \End(\Phi_\bullet ;X_\bullet)\;\;\tilde\to\;\; \Lambda\# H.$$
The first isomorphism is Theorem~\ref{ThmZphiFunctor}, the second morphism corresponds to the interpretation of natural transformations between exact functors as natural transformations in the derived category, and the last isomorphism is given by Lemma~\ref{LemAlgX}.
}
\end{definition}

\begin{lemma}\label{imcom}
The image of~$\zeta_{X_\bullet}$ is contained in $\zeta_\theta(\Lambda),$ with $\zeta_\theta(\Lambda)$ as in Lemma~\ref{LemZeta}\eqref{LemZeta.2}.
\end{lemma}
\begin{proof}
We prove the more general statement that the image of the composition 
$$\mu:\;\End(\Phi_\bullet)\;\to\; \End(\Phi_\bullet; X_\bullet)\;\;\tilde\to\;\; \Lambda\# H$$
is contained in $\zeta_\theta(\Lambda).$
For a natural transformation~$\eta: \Id\Rightarrow \Phi_{h^{\minus 1}}$, we have
$\mu(\eta)=(\Phi_h(\eta_{X_\bullet})\circ \psi_{h^{\minus 1}},h),$ by Lemma~\ref{LemAlgX}.
For the natural transformation~$\Phi_h(\eta):\Phi_h\Rightarrow \Id$
and any morphism $\beta\in \End(X_\bullet)$, we have
$$\beta\circ \Phi_h(\eta_{X_\bullet})\;=\;\Phi_h(\eta_{X_\bullet})\circ \Phi_h(\beta),$$
We set $f:=\Phi_h(\eta_{X_\bullet})\circ \psi_{h^{\minus 1}}$ and use $1_{\Phi_h X_\bullet}=\psi_{h^{\minus 1}}\circ\Phi_h(\psi_h)$ to calculate
$$\label{eqqeeq}\beta\circ f\;=\;\Phi_h(\eta_{X_\bullet})\circ\psi_{h^{\minus 1}}\circ\Phi(\psi_{h})\circ \Phi_h(\beta)\circ\psi_{h^{\minus 1}}\;=\;f\circ \theta_h(\beta).$$
The above implies that the image of~$\mu$ is indeed contained in $\zeta_\theta(\Lambda).$
\end{proof}

\begin{lemma}\label{PropSumm}
For any tilting module $T$ over $A$, considered as an object in $\cD^b(A)$ which admits a $\Phi$-compatible $H$-action~$\psi$, the morphism $\zeta_T$ is injective.
\end{lemma}

\begin{proof}
Lemma~\ref{LemEvT}(i) implies that $\Ev^{\Phi}_T$ is injective.
Since all other morphisms in the composition in Definition~\ref{defzeta} are injective by definition, the statement follows.
\end{proof}

\section{Gradings}\label{Gradings}

We fix an abelian group $G\in \Ab$ for the rest of the paper. As $G$ will be used to define gradings, we adopt the convention to denote its operation by $+$, the identity element by $0$ and the inverse of~$g\in G$ by $-g$.

\subsection{$G$-graded algebras and modules}

\subsubsection{Graded vector spaces}\label{IntroV}
For the group $G$, we introduce the category $\vecc^G_{\mk}$. Its objects are $\mk$-vector spaces $V$ equipped with a $G$-grading, 
$$V=\bigoplus_{g\in G}V_g.$$ The morphisms are those respecting the grading, {\it i.e.} homogeneous $\mk$-linear maps of degree 0. For any $G$-graded $\mk$-vector space~$V$, we write $\partial(v)=g$ for~$v\in V_g$. Whenever $\partial$ is used, we assume that the element on which it acts is homogeneous. 

For any $g\in G$ and a $G$-graded vector space $V$, we define the $G$-graded vector space $\Pi_gV$, which coincides with  $V$ as an ungraded vector space, but with grading given by $(\Pi_g V)_h=V_{h+g}$. For any $v\in V$, we use the notation~$\Pi_g v$ for the element in $\Pi_g V$ identified with $v$ through the equalities $(\Pi_g V)_h=V_{h+g}$. In particular,  \begin{equation}\label{eqConv}v\in V_k \quad\mbox{ implies that }\quad \Pi_g(v) \in  (\Pi_g V)_{k-g}. \end{equation}
In other words, we have $\partial(\Pi_g v)=\partial(v)-g.$

We will interpret $\Pi_g$ as an endofunctor of~$\vecc^G_{\mk}$, defined on a morphism $f:V\to W$ as $\Pi_g(f)(\Pi_g v)=\Pi_g f(v)$, for any $v\in V$. In particular,~$\Pi_0= \Id$ and~$\Pi_{g_1}\Pi_{g_2}= \Pi_{g_1+g_2}$, so the functors $\{\Pi_g\;|\; g\in G\}$ form a group isomorphic to~$G$ and~$\Pi$ is a strict categorical $G$-action on~$\mk$-gmod, in the sense of~\ref{StrictSec}.

\subsubsection{Graded algebras}\label{IntroGA}

A $G$-graded algebra $A$ is a $\mk$-algebra,~$G$-graded as a vector space, such that $A_gA_h\subset A_{g+h}$, for~$g,h\in G$. It follows immediately that $1\in A_0$.
A $G$-graded $A$-module is a $G$-graded $\mk$-vector space $V=\oplus_{g\in G}V_g$ such that the action of~$A$ satisfies $A_gV_h\subset V_{h+g}$. If $A$ is finite dimensional, we define the category $A$-gmod as the category of finite dimensional $G$-graded $A$-modules with morphisms being $A$-linear morphisms of~$G$-graded vector spaces. For $\mk$ as a $G$-graded $\mk$-algebra concentrated in degree zero, $\mk$-gmod is equivalent to~$\vecc_{\mk}^G$. 
Morphism spaces in the category $A$-gmod will be denoted by $\hom_A$.

For any $g\in G$, the functor $\Pi_g$ of \ref{IntroV} induces an endofunctor of~$A$-gmod. Clearly, $\Pi$ yields a strict $G$-action on~$A$-gmod in the sense of~\ref{StrictSec}. The algebras $\End(\Pi;X)$ and $\End(\Pi)$ as in \ref{StrictSec} are then naturally $G$-graded, where for instance $\End(\Pi)_g=\Nat(\Id,\Pi_g)$.

We denote the exact functor forgetting the $G$-grading by
$$F^{\gggg}:\;A\mbox{-gmod}\;\to\; A\mbox{-mod.}$$
When non-essential, we will sometimes leave out reference to this forgetful functor. 
We also identify $F^{\gggg}M$ and $F^{\gggg}\Pi_g M$, for a $G$-graded module $M$ and any $g\in G$.



\begin{lemma}\label{LemPiA}
We have an isomorphism of~$G$-graded algebras
$$\End(\Pi; A)=\bigoplus_{g\in G}\hom_A(A,\Pi_g A)\;\;\tilde\to\; \;A^{\op},$$
where $\alpha\in\hom_A(A,\Pi_g A)$ is mapped to~$\Pi_{\minus g}\alpha(1)$.
\end{lemma}

\begin{proof}
For $\alpha\in\hom_A(A,\Pi_g A)$, we have $\alpha(1)=\Pi_ga$, for some $a\in A_g$.
The described map is thus an isomorphism of~$G$-graded vector spaces. 
Further, for $\alpha\in\hom_A(A,\Pi_g A)$ and $\beta\in\hom_A(A,\Pi_h A)$, their product is
$\alpha\beta=\Pi_h(\alpha)\circ\beta$. Since we have $\Pi_h(\alpha)\circ\beta(1)=\Pi_{g+h}ba$ with $a=\Pi_{\minus g}\alpha(1)$ and $b=\Pi_{\minus h}\beta(1)$, this concludes the proof.
\end{proof}

More generally, we have the following result, which is proved similarly. Set $\cD^{\gggg}:=\cD^b(A\mbox{\rm-gmod})$ and $\cD:=\cD^b(A\mbox{\rm-mod})$.
\begin{lemma}\label{LemFX}
For any $Y_\bullet\in \cD^{\gggg}$, with $\Lambda:=\End_{\cD}(F^{\gggg}Y_\bullet)$, the forgetful functor $F^{\gggg}$ induces an algebra isomorphism
$$\End(\Pi; Y_\bullet)\;\;\tilde\to\;\; \Lambda.$$
\end{lemma}

This lemma thus allows us to equip any endomorphism algebra $\Lambda$ of a gradable object $Y_\bullet$ in $\cD$  with a $G$-grading, where
\begin{equation}\label{GradLambda}\Lambda_g\;\cong\; \Hom_{\cD^{\gggg}}(Y_\bullet, \Pi_gY_\bullet).\end{equation}

\subsubsection{Conventions for gradings}\label{GradConv}
We maintain some conventions for gradings throughout the paper.
\begin{enumerate}[(A)]
\item For two $G$-graded algebras $A,B$, the product $A\otimes B$ is naturally graded, with
$$(A\otimes B)_g\;=\;\bigoplus_{k\in G} A_k\otimes B_{g-k}.$$ 
\item We interpret an ungraded algebra $A$ as graded and concentrated in degree $0$.
\item For an abelian group $H$, the algebra $\mk H$ is $H$-graded, where $(\mk H)_h=\mk h$.
\end{enumerate}

\begin{remark}\label{RemConv}{\rm
Consider $H\in \Ab$ and $A\in \Alg$.
\begin{enumerate}
\item The algebra $A\otimes\mk H$ is $H$-graded using the above conventions.
\item If $A$ is $G$-graded, both $A\otimes \mk H$ and~$A\otimes \mk^H$ are $G$-graded algebras using the above conventions.

\end{enumerate}}

\end{remark}

\subsection{The character group $\hat{G}$ of $G$}


Denote by $\hat{G}\in \Ab$ the $\mk$-character group
$$\hat{G}\;:=\;\Hom_{\Ab}(G,\mk^\times),$$
where multiplication is point-wise. We have a natural group homomorphism
\begin{equation}\label{dualdual}G\mapsto \hat{\hat{G}},\quad g\mapsto \alpha_g, \quad\mbox{with $\alpha_g(\chi)=\chi(g)$, for all $\chi\in \hat{G}$.}\end{equation}

\begin{example}
Assume that $G$ is finite. {\rm It follows that the image of a homomorphism in $\Hom_{\Ab}(G,\mk^\times)$ consists of $|G|$-th roots of unity. }
Assume that $|G|$ is not divisible by $\charr(\mk)$. {\rm This implies all the $|G|$-th roots are different. We thus have
$$\hat{G}=\Hom_{\Ab}(G,\mk^\times)\;\cong\; \Hom_{\Ab}(G,\mC^\times)\;\cong\; \Hom_{\Ab}(G,\mT),$$
where $\mT\cong\mR/\mZ$ is the group of complex numbers of modulus 1. 
In particular, we can identify $\hat{G}$ with the character group in the usual sense, and also with the Pontryagin dual of~$G$ as a locally compact abelian group. In particular, $\hat{G}$ is non-canonically isomorphic to~$G$ and we have orthogonality relations
\begin{equation}\label{orthoR}
\sum_{\chi\in\hat{G}}\chi(g)\chi(-h)\;=\; |G|\delta_{g,h}\quad\mbox{and}\quad \sum_{g\in G}\chi(g)\chi'(-g)\;=\; |G|\delta_{\chi,\chi'}.
\end{equation}
In this case, the group homomorphism in Equation~\eqref{dualdual} is the identity.
}
\end{example}

\begin{example}
Assume that $G=\mZ$. {\rm  We have $\hat{G}=\mG_m=\mk^\times$, the multiplicative group of~$\mk$. In general, this is different from the Pontryagin dual
$$\Hom_{\Ab}(G,\mT)\;\cong\;\mT$$
of $G=\mZ$ as a locally compact abelian group.}
\end{example}

\begin{lemma}\label{LemAGAG}
The algebra morphism 
$\mk\hat{G}\;\to\;\mk^G$ given by interpreting characters as elements of~$\Hom_{\Set}( G, \mk)$ is injective 
 and
 an isomorphism if $|G|$ is finite and is not divisible by $\charr(\mk)$.
\end{lemma}
\begin{proof}
We have an injective morphisms of monoids
$$\hat{G}\;=\;\Hom_{\Ab}(G,\mk^\times)\;\hookrightarrow\;\Hom_{\Set}(G,\mk)\;=\;\mk^G,$$
which thus leads to an algebra morphism $\mk\hat{G}\;\to\;\mk^G$. This morphism is injective by Dedekind's result on linear independence of characters, see e.g. \cite[Proposition~4.30]{Rotman}.

Now assume that $|G|$ is finite and is not divisible by $\charr(\mk)$. The map
\begin{equation}\label{eqInver}\mk^{ G}\;\to\; \mk\hat{G};\quad f\;\mapsto\; \frac{1}{|G|}\sum_{\eta\in\hat{G}}\left(\sum_{l\in G}\eta(-l)f(l)\right)\eta\end{equation}
is an inverse, as follows from a direct computation using Equations~\eqref{orthoR}.
\end{proof}


\subsection{Actions versus gradings}\label{AlgAct}
For $V\in \vecc_{\mk}^G$ and $\chi\in \hat{G}$, we define $\psi_{\chi}\in\End_{\mk}(V)$ by $\psi_\chi(v)=\chi(\partial v)v$. It follows that $(V,\psi)\in \Rep_{\mk}\hat{G}$.

\begin{proposition}\label{fffun}
${}$
\begin{enumerate}[(i)]
\item Interpreting $V\in \vecc_{\mk}^G$ as an element of~$\Rep_{\mk}\hat{G}$ as above yields a faithful functor 
$$\Xi\;:\;\;\vecc_{\mk}^G\;\to\; \Rep_{\mk}\hat{G},\quad V\mapsto (V,\psi).$$
\item If $V\in\vecc^G_{\mk}$ is a $G$-graded algebra, then $\psi$ is an $H$-action on the algebra $V$.
\item If $A$ is a $G$-graded algebra and $V\in \vecc^G_{\mk}$ is a graded $A$-module, then the actions on $\Xi(A)$ and $\Xi(V)$ are compatible. 
\end{enumerate}
\end{proposition}





For $V\in \vecc^G_{\mk}$ , we simply write $v_\chi$, for~$\psi_\chi(v)=\chi(\partial v)v$. The lemma thus implies, in particular, that, for a $G$-graded algebra $A$, we have
a group homomorphism
\begin{equation}\label{GAut}\phi:\hat{G}\to \Aut(A),\quad \phi_\chi(a)=a_\chi,\qquad\mbox{for all $\chi\in\hat{G}$ and $a\in A$.}\end{equation}

\begin{lemma}\label{LemGfin}
When $|G|$ is finite and not divisible by $\charr(\mk)$, $\Xi$ in Proposition~\ref{fffun}(i) is an equivalence of categories, which restricts to an equivalence between $G$-graded algebras and algebras with $\hat{G}$-action.
\end{lemma}

\begin{proof}
The inverse to~$\Xi$ is constructed using Equations~\eqref{orthoR}.
\end{proof}

Under the conditions of Lemma~\ref{LemGfin}, we thus find that the theory of $G$-gradings is equivalent to that of $\hat{G}$-actions as in Section~\ref{GroupAc}. In general, the theory of $\hat{G}$-actions will be much richer. In particular, $\Rep_{\mk}\hat{G}$ is far from being semisimple, contrary to~$\vecc^G_{\mk}$.

\begin{remark}{\rm
When $|G|$ is not finite or divides $\charr(\mk)$, the correct analogue of the equivalence in Lemma~\ref{LemGfin} is the well-known statement that we have an equivalence of categories
$$\vecc^G_{\mk}\;\stackrel{\sim}{\to}\; \Rep\,\mH,$$
for the (diagonalisable) affine group scheme $\mH:=\Spec\, \mk G$. Note that, by definition, $\Rep\,\mH$ is the category of comodules over the Hopf algebra $\mk[\mH]:=\mk G$. It then follows that the group of~$\mk$-points of the group scheme $\mH$ is
$$\mH(\mk):=\Hom_{\Alg}(\mk[\mH],\mk)\;=\;\Hom_{\Alg}(\mk G,\mk)\;=\;\Hom_{\Ab}(G,\mk^\times)\;=\;\hat{G}.$$
However, the canonical functor,
$$\Rep \,\mH\;\to \;\Rep_{\mk}\mH(\mk)\;=\;\Rep_{\mk}\hat{G}$$
is neither full nor dense in general.

For $G=\mZ$ and $\charr(\mk)=0$, the above functor, and hence $\vecc^{\mZ}_{\mk}\to\Rep_{\mk}\mG_m$ in Proposition~\ref{fffun}(i), is fully faithful, but not dense. When $G=\mZ_2$ and $\charr\mk=2$, the functor is dense but not full.}
\end{remark}

\subsection{The extended centre for a $G$-grading}
Fix a finite dimensional unital associative $G$-graded $\mk$-algebra $A$. Consider the algebra $A\otimes\mk\hat{G}$ with the $G$-grading of Remark~\ref{RemConv}(2) and the $\hat{G}$-grading 
of Remark~\ref{RemConv}(1). This actually yields a $G\times\hat{G}$-grading. 

\begin{schol}\label{DefZGA}
We apply Definition \ref{DefExtCen} to the $\hat{G}$-action~$\phi$ in Equation \eqref{GAut}.

{\rm
\begin{enumerate}[$($i$)$]
\item\label{DefZGA.1} The algebra $\cZ^\phi(A)$ is the $G\times \hat{G}$-graded 
subalgebra of~$A\otimes\mk\hat{G}$, where, for given $g\in G$ and $\chi\in\hat{G}$, the space
$\cZ^\phi(A)_{(g,\chi)}$ is spanned by all $(x,\chi)$, for which 
$x\in A_g$ and
$$x\,y=y_\chi\, x,\qquad\mbox{for all $y\in A$}.$$
\item\label{DefZGA.2} Consider the algebra morphism $A\otimes\mk\hat{G}\tto A$ given by $a\otimes\chi\mapsto a$.
The image of~$\cZ^\phi(A)$ under $A\otimes\mk\hat{G}\tto A$ is denoted by $\underline{\cZ}^\phi(A)$. The algebra  $\underline{\cZ}^\phi(A)$ is still naturally $G$-graded, but will, in general, no longer be $\hat{G}$-graded, see Example~\ref{ExDual}.
\item\label{DefZGA.3} By Proposition~\ref{fffun}(ii), the $\hat{G}$-grading on $\cZ^{\phi}(A)$ yields a $\hat{\hat{G}}$-action. By Equation~\eqref{dualdual}, we can pull this back to a $G$-action, where $g$ acts on $(x,\chi)\in \cZ^\phi(A)$ by sending it to~$(\chi(g)x,\chi)$.
\end{enumerate}
}
\end{schol}

\begin{remark}
{\rm
Most of the multiplication in the algebra $\cZ^\phi(A)$ is zero. Consider $g,h\in G$ and~$x\in A_g$,~$y\in A_h$ such that the elements $(x,\chi),(y,\chi')\in A\otimes\mk\hat{G}$ belong to~$\cZ^G(A)$. Then, clearly, 
$(x,\chi)(y,\chi')=0$ unless $\chi'(g)\chi(h)=1$.
}
\end{remark}


\section{The $G$-centre}\label{SecGcentre}

Fix a finite dimensional unital associative $G$-graded $\mk$-algebra $A$. We denote elements of the algebra $A\otimes \mk^G=\Hom_{\Set}(G,A)$ as
$$\mathbf{x}:G\to A,\quad g\mapsto x^{(g)}.$$

\begin{definition}\label{DefGcentre}{\rm
The {\em $G$-centre} $\cZ^G(A)$ of $A$ is the $G$-graded subalgebra of $A\otimes \mk^G$ given by
$$\{\xx\in A\otimes \mk^G\;|\; x^{(g)}y=yx^{(g+h)},\quad\mbox{for all }\; y\in A_h \;\mbox{and }h\in G\}.$$
The algebra $\cZ^G(A)$ admits a $G$-action, where the element $k\in G$ acting on $\xx$ yields $\{g\mapsto x^{(k+g)}\}$.
The algebra $\underline{\cZ}^G(A)$ is the image of $\cZ^G(A)$ under the morphism $A\otimes \mk^G\tto A$ given by $\xx\mapsto x^{(0)}$.}
\end{definition}

We can express the $G$-centre naturally in a generalisation of~\eqref{CentreId}.
Contrary to the previous generalisation of~\eqref{CentreId} to~$\cZ^\phi(A)$ in Theorem~\ref{ThmZphiFunctor}, we use the category $A$-gmod instead of~$A$-mod. 

\begin{theorem}\label{DescCen}
As $G$-graded algebras, we have
$\cZ^G(A)^{\op}\;\cong\;\End(\Pi).$
\end{theorem}

This theorem will be proved in the following subsection. First we demonstrate that, when $|G|$ is finite and not divisible by $\charr(\mk)$ and hence $G$-gradings can be identified with $\hat{G}$-actions, the $G$-centre $\cZ^G(A)$ is isomorphic to the extended centre $\cZ^\phi(A)$ for the $\hat{G}$-action $\phi$ on $A$. Under these conditions, the $G$-action on $\cZ^G(A)$ must also correspond to a $\hat{G}$-grading, given by
\begin{equation}\label{eqGhatgrad}(\cZ^G(A))_{\chi}\;=\;\{\xx\in \cZ^G(A)\,|\, x^{(g)}=\chi(g)x^{(0)},\mbox{ for all $g\in G$}\}. \end{equation}

\begin{proposition}\label{PropphiG}
The injective morphism $A\otimes\mk\hat{G}\hookrightarrow A\otimes \mk^G$ which follows from Lemma~\ref{LemAGAG} restricts to an injective morphism of $G$-graded algebras
$$\cZ^\phi(A)\;\hookrightarrow \cZ^{G}(A),$$
which intertwines the $G$-actions in Scholium~\ref{DefZGA}\eqref{DefZGA.3} and Definition~\ref{DefGcentre}.
This is an isomorphism of $G\times\hat{G}$-graded algebras when $|G|$ is finite and not divisible by $\charr(\mk)$.
\end{proposition}
\begin{proof}
By definition, $(x,\chi)\in\cZ^\phi\subset A\otimes\mk\hat{G}$, as in Scholium~\ref{DefZGA}(i), is sent to 
$$\xx:G\to A,\qquad g\mapsto \chi(g)x,$$
which is, clearly, an element of $\cZ^G(A)$. Since the $G$-gradings of both algebras are immediately inherited from the one on $A$, it is obvious that this morphism respects the $G$-grading. Equation~\eqref{eqGhatgrad} further implies that the image of $(x,\chi)\in \cZ^\phi(A)_\chi$ is indeed in $\cZ^G(A)_\chi$.

When $|G|$ is finite and not divisible by $\charr(\mk)$, one checks similarly that the inverse $A\otimes\mk^G\to A\otimes\mk\hat{G}$ in Equation~\eqref{eqInver} maps $\cZ^G(A)$ to~$\cZ^{\phi}(A)$.
\end{proof}

\begin{remark}
It follows similarly from the definitions that we obtain a morphism $\underline{\cZ}^\phi(A)\;\hookrightarrow \underline{\cZ}^{G}(A),$ which is an isomorphism when $|G|$ is finite and not divisible by $\charr(\mk)$.
\end{remark}

\subsection{Evaluation}
We study the evaluation in Lemma~\ref{LemEva}
$$\Ev_{M}^{\Pi}:\; \End(\Pi)\;\to\;\End(\Pi; M),$$
and the astute evaluation of Definition~\ref{DefEvPi},
$$\Delta\Ev^{\Pi}_{M}:\; \End(\Pi)\;\to\; \Hom_{\Set}(G,\End(\Pi; M)).$$

First we apply $\Delta\Ev^{\Pi}$ to the left regular module $M={}_AA$. By Lemma~\ref{LemPiA}, we have an isomorphism 
$$\Hom_{\Set}(G,\End(\Pi; A))\;\cong\;\Hom_{\Set}(G,A^{\op})\;\cong\; A^{\op}\otimes \mk^G.$$ We denote by $\Delta\overline{\Ev}_A^{\Pi}$ the composition of~$\Delta{\Ev}_A^{\Pi}$ with this isomorphism.

\begin{proposition}\label{PropEvPi}
The astute evaluation morphism
$$\Delta\overline{\Ev}_A^{\Pi}:\;\End(\Pi)\;\to\;  A^{\op}\otimes \mk^G=(A\otimes \mk^G)^{\op}$$
\begin{equation}\label{eqNatTrProof}\Nat(\Id,\Pi_g)\ni\eta\;\mapsto\; \{k\mapsto \Pi_{\minus g\minus k}(\eta_{\Pi_{k}A}(\Pi_k1)),\;\mbox{$k\in G$}\}.\end{equation}
is injective and has $(\cZ^G(A))^{\op}$ as the image. 
\end{proposition}

\begin{proof}
The injectivity of~$\Delta\overline{\Ev}_A^{\Pi}$ is obvious because the functors $\Pi_g$ are exact and any object in $A$-gmod is a factor module of a finite direct sum of modules isomorphic to~$\Pi_kA$, $k\in G$.

In the remainder of the proof, any multiplication of elements in $A$ will 
be interpreted as multiplication inside $A$, never in $A^{\op}$.

Now consider a natural transformation
$\eta: \Id\Rightarrow \Pi_g$ and $\xx=\Ev_{M}^{\Pi}(\eta)$, with $x^{(k)}= \Pi_{\minus g\minus k}(\eta_{\Pi_{k}A}(\Pi_k1))$ as in Equation~\eqref{eqNatTrProof}.
Consider arbitrary $h\in G$ and $a\in A_h$. This $a$ defines, for all $l\in G$, 
a morphism $\alpha_l:\Pi_l A\to \Pi_{l+h}A$ given by $\Pi_l b\mapsto \Pi_{l+h}ba$, 
for all $b\in A$. Note that, by definition, $\Pi_{l'}(\alpha_l)=\alpha_{l+l'}$. 
Since $\eta$ is a natural transformation, we have a commuting diagram
$$\xymatrix{\Pi_k A\ar[r]^{\eta_{\Pi_k A}}\ar[d]^{\alpha_k}&\Pi_g\Pi_kA\ar[d]^{\alpha_{g+k}}\\
\Pi_{h+k}A\ar[r]^{\eta_{\Pi_{h+k}A}}& \Pi_g\Pi_{h+k}A},$$
meaning that $x^{(k)}a=ax^{(h+k)}$, or $\xx\in\cZ^G(A)$. This implies that the image of 
$\Delta\overline{\Ev}_A^{\Pi}$ is contained in $(\cZ^G(A))^{\op}.$

Now, start from an arbitrary $\xx\in \cZ^G(A)_g$, for~$g\in G$. We want to define a natural transformation
$\eta: \Id \Rightarrow \Pi_g$. For any $M\in A$-gmod, we define a morphism
$$\eta_M: M\to \Pi_g M\; \;\;\mbox{by}\; \;\; v\mapsto \Pi_gx^{(-h)}v,\;\;\mbox{ for }\; v\in M_h.$$
This morphism is $A$-linear by construction.
For any morphism $\alpha:M\to N$, we claim that  $\eta_N\circ \alpha=\Pi_g(\alpha)\circ \eta_M$. Indeed, for~$v\in M_h$, we have
$$\eta_N\circ \alpha(v)=\Pi_g x^{(-h)}\alpha(v)=\Pi_g \alpha(x^{(-h)}v)=\Pi_g(\alpha)\left(\Pi_g x^{(-h)}v\right)=\Pi_g(\alpha)\circ \eta_M (v),$$
so $\eta$ is a natural transformation. Thus we find that the image of 
$\Delta\overline{\Ev}_A^{\Pi}$ is, in fact, equal to~$(\cZ^G(A))^{\op},$ concluding the proof.
\end{proof}

Proposition~\ref{PropEvPi} implies Theorem~\ref{DescCen}. Additionally, we also have the following two corollaries.
First, we compose $\Ev_A^\Pi$ with the isomorphism in Lemma~\ref{LemPiA}.

\begin{corollary}
The image of
$\;\overline{\Ev}^{\Pi}_A:\;\End(\Pi)\to A^{\op}$
is given by $\underline{\cZ}^G(A)^{\op}$.
\end{corollary}

\begin{proof}
By definition, we have a commuting triangle of algebra morphisms
$$\xymatrix{
\End(\Pi)\ar[rr]^{\Delta\overline{\Ev}^{\Pi}_A}
\ar[rrd]_{\overline{\Ev}^{\Pi}_A}&&A^{\op}\otimes \mk^G\ar@{->>}[d]\\
&& A^{\op}
}$$
in which the vertical arrow is given by $\xx\mapsto x^{(0)}$. The result hence follows from Proposition~\ref{PropEvPi} and Definition~\ref{DefGcentre}.
\end{proof}

\begin{corollary}\label{CorGcentrePart}
Assume that $|G|$ is finite and not divisible by $\charr(\mk)$.
Consider the $G\times \hat{G}$-grading on~$\cZ^G(A)$ as given by  Definition~\ref{DefGcentre} and Equation~\eqref{eqGhatgrad}. 
The algebra isomorphism in Theorem~\ref{DescCen} restricts to vector space isomorphisms
$$\cZ^G(A)_{g,\chi}\;\cong\; \{\eta\in \Nat(\Id,\Pi_g)\;|\; \eta_{\Pi_k}\,=\, \chi(k)\,\Pi_k(\eta),\;\;
\text{ for all } k\in G\}.$$
\end{corollary}

\begin{proof}
Consider $\eta$ as in the right-hand side. By Equation~\eqref{PropEvPi}, we have that the corresponding $\xx\in \cZ^G(A)$ is given by
$$x^{(k)}\;:=\;\Pi_{\minus g\minus k}(\eta_{\Pi_{k}A}(\Pi_k1)).$$
By assumption, we have 
$$\eta_{\Pi_kA}(\Pi_k1)= \chi(k)\Pi_k(\eta_A)(\Pi_k1)= \chi(k)\Pi_k(\eta_A(1)),$$
which means
$$x^{(k)}\;:=\;\chi(k)\Pi_{-g}(\eta_A(1))\;=\;\chi(k)x^{(0)}.$$
Since $\Pi_{-g}(\eta_A(1))\in A_g$, Equation~\eqref{eqGhatgrad} shows that $\xx\in \cZ^G(A)_{g,\chi}$. 
\end{proof}

In analogy with Definition~\ref{defzeta}, we introduce the following composition of morphisms. We set $\cD^{\gggg}=\cD^b(A\mbox{-gmod})$ and  $\cD=\cD^b(A\mbox{-mod})$.

\begin{definition}\label{DefZetaDelta}
{\rm
Consider $X_\bullet\in\cD^{\gggg}$, with $\Lambda:=\End_{\cD}(X_\bullet)$ equipped with the $G$-grading inherited in Lemma~\ref{LemFX} and Equation~\eqref{GradLambda}. The morphism $$\Delta\zeta_{X_\bullet}: \,\cZ^G(A)^{\op}\;\to\; \Lambda\otimes \mk^G$$ of~$G$-graded algebras is given by the composition 
$$\cZ^G(A)^{\op}\;\;\tilde\to\;\; \End(\Pi)\;\hookrightarrow\;\End(\Pi_\bullet)\;\to\;\Hom_{\Set}(G,\End(\Pi_\bullet; X_\bullet))\;\tilde\to\; \Lambda\otimes\mk^G.$$
The first isomorphism is Proposition~\ref{PropEvPi}, the third morphism is $\Delta\Ev_{X_\bullet}^{\Pi}$ in Definition~\ref{DefEvPi} and the last isomorphism is induced from the one in Lemma~\ref{LemFX}.
}
\end{definition}

\begin{lemma}\label{LemZetaIm}
With notation as in Definition~\ref{DefZetaDelta}, the image of~$\Delta\zeta_{X_\bullet}$ is contained in $\cZ^G(\Lambda^{\op})^{\op}$.
The corresponding morphism
$$\Delta\zeta_{X_\bullet}\;:\cZ^G(A)^{\op}\;\to\; \cZ^G(\Lambda^{\op})^{\op}$$
is a morphism of~$G\times\hat{G}$-graded algebras.
\end{lemma}

\begin{proof}
The image under $\Delta\zeta_{X_\bullet}$ of an element in $\cZ^G(A)$ corresponding to the natural transformation~$\eta:\Id\Rightarrow\Pi_g$ is given by
$$\xx\in \Lambda\otimes\mk^G,\quad\mbox{ with $x^{(k)}:=F^{\gggg}(\eta_{\Pi_k X_\bullet})$}.$$
For an arbitrary $\beta\in \Hom_{\cD^{\gggg}}(X_\bullet,\Pi_h X_\bullet)$, the fact that $\eta$ is a natural transformation implies that
$$\Pi_{g+k}(\beta)\circ\eta_{\Pi_kX_\bullet}\;=\;\eta_{\Pi_{k+h}X_\bullet}\circ\Pi_k(\beta).$$
In particular, we have
$$F^{\gggg}(\beta)\circ x^{(k)}=x^{(k+h)}\circ F^{\gggg}(\beta),$$
which proves that $\xx$ is in $\cZ^G(\Lambda^{\op})$.

That the $G$-grading is preserved follows by construction. Now, take an element in $\cZ^G(A)_{g,\chi}$, 
for~$g\in G$ and $\chi\in\hat{G}$. By Corollary~\ref{CorGcentrePart}, this corresponds to a natural 
transformation $\eta:\Id\Rightarrow\Pi_g$ satisfying $\eta_{\Pi_k}= \chi(k)\,\Pi_k(\eta)$, for all $k\in G$. 
Therefore 
$$x^{(k)}=F^{\gggg}(\eta_{\Pi_k X_\bullet})=\chi(k)F^{\gggg}(\Pi_k(\eta_{ X_\bullet}))=\chi(k)F^{\gggg}(\eta_{ X_\bullet})=\chi(k)x^{(0)},$$
so $\xx\in \cZ^G(A)_\chi$, by Equation~\eqref{eqGhatgrad}.
 This completes the proof.
\end{proof}

\subsection{The $G$-centre and Gradable derived equivalences}

Following~\cite[Section~3.2]{CM4}, we use the term ``gradable derived equivalence'' 
for an equivalence which commutes  both with grading shifts and the suspension functor.

\begin{definition}
{\rm 
Consider two $G$-graded algebras $A$ and~$B$.
\begin{enumerate}[(i)]
\item A {functor} $H:\cD^b(A\mbox{{\rm-gmod}})\to\cD^b(B\mbox{{\rm-gmod}})$ is {\em graded} if it intertwines the $G$-actions $\Pi$, as in~\ref{CatComm}. 
\item A {gradable derived equivalence} between two $G$-graded algebras $A$ and $B$ is a graded and triangulated functor $F:\cD^b(A\mbox{{\rm-gmod}})\to\cD^b(A\mbox{{\rm-gmod}})$ which admits an inverse which is also a graded and triangulated functor. A {gradable derived equivalence} is {\em strong} if it is strong in the sense of Section~\ref{NotConv}.
\end{enumerate}
}
\end{definition}

The following is a generalisation of \cite[Proposition~9.2]{Derived1} to~$G$-graded algebras and an analogue of Theorem~\ref{ThmDerEqPhi}.

\begin{theorem}\label{ThmGradEq}
If two $G$-graded algebras $A$ and $B$ are strongly gradable derived equivalent, then $\cZ^G(A)\cong\;\cZ^G(B)$ as $G\times \hat{G}$-graded algebras.
\end{theorem}
\begin{proof}
Let $F:\cD^b(A\mbox{-gmod})\to\cD^b(B\mbox{-gmod})$ denote a gradable derived equivalence. We will write $\cD^{\gggg}$ for~$ \cD^b(B\mbox{-gmod})$.
We set $X_\bullet\in \cD^{\gggg}$ equal to~$F(A)$. By Lemmata~\ref{LemInt2} and~\ref{LemPiA}, we have algebra isomorphisms
$$\End(\Pi_\bullet ;X_\bullet)\;\cong\; \End(\Pi; A)\;\cong\; A^{\op}.$$
as $G$-graded algebras. By Lemma~\ref{LemZetaIm}, we then have a morphism of~$G\times \hat{G}$-graded algebras
$$\Delta\zeta_{X_\bullet}:\; \cZ^{G}(B)^{\op}\;\to\; \cZ^G(A)^{\op}.$$
This morphism is injective by Lemma~\ref{LemEvT}(ii).

By symmetry in the definition of gradable derived equivalences, the fact that the injective morphisms respect the $G$-grading and the fact that $A$ is finite dimensional, it follows that the injective morphisms must be bijections.\end{proof}


\section{Superalgebras}\label{SecSuper}

We consider the special case $G=\mZ_2=\{\oa,\ob\}$ and we assume~$\charr(\mk)\not=2$. $G$-graded algebras are then also known as superalgebras and the category $A$-gmod is known as the category of supermodules.

\subsection{Super, anti and ghost centre}

The character group is $\hat{G}=\{\chi_0,\chi_1\}$, where $\chi_0(\ob)=1$ and~$\chi_1(\ob)=-1$.
For the interpretation of~$G$-graded algebras as superalgebras, some terminology appeared in \cite{Gorelik}, which we link to our constructions.

The {\em super centre} of~$A$, denoted by $s\cZ(A)$, is the subalgebra of~$A$ spanned by homogeneous elements $x$ satisfying
\begin{equation}\label{SAeq1}xy=(-1)^{\partial x\,\partial y}yx,\end{equation}
for all homogeneous $y\in A$. The {\em anti centre}, denoted by $a\cZ(A)$, is a subspace of~$A$ spanned by homogeneous elements $x$ satisfying
\begin{equation}\label{SAeq2}xy=(-1)^{(\partial x+1)\partial y}yx.\end{equation}
Generally, the anti centre does not constitute a subalgebra. The product of two elements of~$a\cZ(A)$ belongs to~$s\cZ(A)$. The subalgebra of~$A$ consisting of linear combinations of elements of the super and the anti centre is known as the {\em ghost centre},~$\widetilde\cZ(A)=s\cZ(A)+a\cZ(A)$.

We can rewrite Equation~\eqref{SAeq1} as
 $$xy=\begin{cases}y_{\chi_0}\,x,&\mbox{if $x\in A_{\oa}$;}\\y_{\chi_1}\,x,&\mbox{if $x\in A_{\ob}$}.\end{cases}$$
 Similarly, Equation~\eqref{SAeq2} becomes
  $$xy=\begin{cases}y_{\chi_1}\,x,&\mbox{if $x\in A_{\oa}$;}\\y_{\chi_0}\,x,&\mbox{if $x\in A_{\ob}$}.\end{cases}$$
By Proposition~\ref{PropphiG} and Scholium~\ref{DefZGA}(i), we thus have the following.
 \begin{proposition}\label{PropS2}
For $G=\mZ_2$, the $G\times \hat{G}$-grading of~$\cZ^G(A)$ satisfies
\begin{enumerate}[$($i$)$]
\item\label{PropS2.1} $s\cZ(A)\;=\;\cZ^G(A)_{\oa,\chi_0}\,\oplus \, \cZ^G(A)_{\ob,\chi_1};$
\item\label{PropS2.2} $a\cZ(A)\;=\;\cZ^G(A)_{\oa,\chi_1}\,\oplus\, \cZ^G(A)_{\ob,\chi_0}.$
\end{enumerate}
As vector spaces, we hence have $$\cZ^G(A)=s\cZ(A)\,\oplus\,a\cZ (A),$$ where the latter direct sum is abstract, not inside $A$.
\end{proposition}

Scholium~\ref{DefZGA}(ii) then yields the following.
\begin{proposition}\label{PropS1}
For $G=\mZ_2$, the ghost centre $\widetilde{\cZ}(A)$ is equal to~$\underline{\cZ}^G(A)$. In particular, as subalgebras of~$A$, we have
$$\underline{\cZ}^G(A)\;=\;s\cZ(A)\,+\,a\cZ(A).$$
\end{proposition}

We end this subsection with an example in which we demonstrate 
all the above notions for a small $\mZ_2$-graded algebra.

\begin{example}\label{ExDual}{\rm
Consider the algebra $A:=\mk[x]/(x^2)$ of dual numbers. We set $A=\mk\oplus \mk x$ and consider $A$ as a $\mZ_2$-graded algebra with $A_{\oa}=\mk$ and~$A_{\ob}=\mk x$. 
We have $\cZ^G(A)_{\chi_0}=\cZ(A)=A$ and~$\cZ^G(A)_{\chi_1}=A_{\ob}$. Clearly $\underline{\cZ}^G(A)=A$ does not inherit the $\hat{G}$-grading.
It follows that $s\cZ(A)=A$ and~$a\cZ(A)=A_{\ob}$.
}
\end{example}

\subsection{Derived equivalences of superalgebras}\label{superequiv}
For a superalgebra $A$, we set $\Pi_{\oa}=\Id$ as usual, and~$\bpi:=\Pi_{\ob}$. The category $A$-gmod is then a $\bpi$-category in the sense of \cite[Definition~1.6(i)]{supercategory}.

Let $A$ and $B$ be superalgebras. According to \cite[Definition~1.6(ii)]{supercategory}, a {\em $\bpi$-functor} in our setting is a functor $F$ from~$A$-gmod to~$B$-gmod, or their derived categories, with a fixed natural isomorphism $\xi^F:\bpi\circ F\Rightarrow F\circ\bpi$ such that $\xi^F_{\bpi}\circ \bpi(\xi^F)$ equals the identity natural transformation of~$F$, when interpreted using $\bpi^2=\Id$. 
We thus conclude that $F$ is a $\bpi$-functor if and only if $F$ intertwines the $\Pi$-actions as in~\ref{CatComm}.

\begin{theorem}
Let $A,B$ be superalgebras and $F:\cD^b(A)\to\cD^b(B)$ be a strong triangulated $\bpi$-equivalence which admits a strong triangulated $\bpi$-functor as inverse.
Then we have algebra isomorphisms
$$s\cZ(A)\cong s\cZ(B)\quad\mbox{and}\quad s\cZ(A)\oplus a\cZ(A)\cong s\cZ(B)\oplus a\cZ(A).$$
\end{theorem}

\begin{proof}
By Theorem~\ref{ThmGradEq}, we have an equivalence of~$G\times \hat{G}$-graded algebras $\cZ^G(A)\cong\cZ^G(B)$. The conclusions thus follow from Proposition~\ref{PropS2}.
\end{proof}

This implies that, under appropriate derived equivalences of superalgebras, the super centre 
is preserved, as well as the exterior sum of the super and the anti centre. Whether the ghost 
centre is also preserved does not follow from the general theory.

\subsection{Alternative categorical realisations of the supercentre}

\subsubsection{Supernatural transformations}

For a $\mZ_2$-graded algebra $A$, we introduce {\em the supercategory of modules} $\cC\,=A$-smod. 
This $\mk$-linear category has the same objects as $A$-gmod, but larger spaces of homomorphisms. 
For two graded modules $M,N$, the space of morphism $\Hom_{\cC}(M,N)$ in $A$-smod is the 
$\mZ_2$-graded vector space, with $\Hom_{\cC}(M,N)_{\oa}=\hom_A(M,N)$ (the $A$-module 
morphism respecting the grading) and~$\Hom_{\cC}(M,N)_{\ob}$, the elements $f$ of 
$$\Hom_{\mk}(M_{\oa},N_{\ob})\oplus \Hom_{\mk}(M_{\ob},N_{\oa})\subset\Hom_{\mk}(M,N),$$ 
which satisfy $f(av)=(\minus 1)^{\partial a}af(v)$, for homogeneous $a\in A$ and~$v\in M$. 
The category $A$-smod, contrary to~$A$-mod and~$A$-gmod, will  not be abelian in general. 

We have
$$\End_{A\minus{\rm{smod}}}(A)\;\cong\; A^{\sop},$$
with $A^{\sop}$ the superalgebra with underlying vector space $A$ and multiplication given by
$$m(a,b)\;=\;(-1)^{\partial a\,\partial b}ba.$$
We, clearly, have
$$s\cZ(A^{\sop})\;\cong\; s\cZ(A)\;\cong\;s\cZ(A)^{\sop}.$$

Following, \cite[Definition~(1.1)]{supercategory}, a {\em supercategory}, resp. {\em superfunctor}, 
is a category, resp. functor, enriched over the category $\vecc^{\mZ_2}_{\mk}$.
The category $A$-smod is an example of a supercategory. We recall the notion of {\em supernatural transformations},
from
\cite[Definition~(1.1)(iii)]{supercategory}. The space $\SNat(F,G)_{\oa}$ is 
spanned by all natural transformations $\eta:F\Rightarrow G$ such that $\eta_M$ is even for each $M\in A$-smod. 
An element of~$\SNat(F,G)_{\ob}$ is a family of odd morphisms 
$\{\eta_M,M\in A\mbox{-smod}\}$ in $A$-smod such that $\eta_N\circ f=(\minus 1)^{\partial f}f\circ\eta_M$, 
for any $f:M\to N$.
 
\begin{proposition}
With $\Id$ the identity functor in $A${\rm -smod}, we have an isomorphism of superalgebras
$$\End(\Id)=\SNat(\Id,\Id)\;\cong\;s\cZ(A).$$
\end{proposition}

\begin{proof}
We consider the ordinary evaluation 
$$\End(\Id)\;\to\; \End_{A\minus{\rm{smod}}}(A)\cong A^{\sop}.$$
Since, for any $M$ in $A$-smod and $v\in M$, there exists $\alpha\in \Hom_{A\minus{\rm{smod}}}(A,M)$ with $v\in \im(\alpha)$, this evaluation is injective.

A homogeneous supernatural transformation $\eta:\Id\Rightarrow\Id$ satisfies
$$\eta_A\circ\alpha\;=\;(-1)^{\partial\alpha\,\partial\eta}\alpha\circ\eta_A,$$
for each homogeneous morphism $\alpha:A\to A$.
We set $a:=\eta_A(1)$. The above equation then implies that $a\in s\cZ(A)$. Every supernatural transformation thus yields an element of the supercentre.

Now we start from a homogeneous $a\in s\cZ(A)$ and define, for each module $M$, morphisms $\eta_M\in\End_{A\minus{\rm{smod}}}(M)$ by
$$\eta_M(v)=av.$$
These form a supernatural transformation, completing the proof.
\end{proof}

\subsubsection{$\bpi$-natural transformations} We return to the category $A$-gmod.

Recall the notion of~$\bpi$-functors on $A$-gmod from Subsection~\ref{superequiv}. 
We follow the convention where $\Id$ and~$\bpi$ are $\bpi$-functors where $\xi^{\Id}$ 
is the identity and~$\xi^{\bpi}$ minus the identity. Following 
\cite[Definition~1.6(iii)]{supercategory}, a $\bpi$-natural transformation between 
two $\bpi$-functors $F$ and~$K$ on~$A$-gmod, is a natural transformation~$\eta:F\Rightarrow K$ 
such that 
$$\eta_{\bpi}\circ\xi^F=\xi^K\circ \bpi(\eta),$$
inside $\Nat(\bpi\circ F,K\circ\bpi)$.
We let $\Nat^{\bpi}$ denote the spaces of~$\bpi$-natural transformations. 
The subspace of~$\End(\Pi)$ given by
$$\Nat^{\bpi}(\Id,\Id)\oplus 
\Nat^{\bpi}(\Id,\bpi)$$
constitutes a subalgebra, which we denote by $\End^{\bpi}(\Pi)$.

\begin{proposition}
We have an isomorphism of superalgebras
$$\End^{\bpi}(\Pi)\;\cong\;s\cZ(A)^{\op}.$$
\end{proposition}
\begin{proof}
As an immediate consequence of Theorem~\ref{DescCen} and Corollary~\ref{CorGcentrePart}, we have
$$\End^{\bpi}(\Pi)\;=\;\End(\Pi)_{\oa,\chi_0}\,\oplus\, \End(\Pi)_{\ob,\chi_1}\;\cong\; \cZ^G(A)_{\oa,\chi_0}^{\op}\,\oplus\, \cZ^G(A)_{\ob,\chi_1}^{\op}.$$
The result then follows from Proposition~\ref{PropS2}(i).
\end{proof}


\section{$G$-Hochschild cohomology speculations}\label{SecHoch}

By Theorems~\ref{ThmZphiFunctor} and~\ref{DescCen}, it is natural to introduce the following spaces for an algebra $A$ with an $H$-action, respectively a $G$-grading:
\begin{itemize}
\item $\Ext^\bullet(\Phi):=\bigoplus_{i,h}\Ext^i(\Id,\Phi_h)$;
\item $\Ext^{\bullet}(\Pi):=\bigoplus_{i,g}\Ext^i(\Id,\Pi_g)$;
\end{itemize}
where the first extension groups are taken in the category $\Fun(A\mbox{{\rm -mod}})$ and the second in $\Fun(A\mbox{{\rm -gmod}})$.
These can be interpreted as generalisations of Hochschild cohomology, see e.g.~\cite[Chapter~7]{He}. The spaces can again be given the structure of algebras, using the approach of~\ref{StrictSec} and the Yoneda product. 

Based on Proposition~\ref{PropphiG} and Theorems~\ref{ThmDerEqPhi} and~\ref{ThmGradEq} and \cite[Proposition~2.5]{Derived2}, we arrive at the following natural questions:
\begin{enumerate}
\item Consider a $G$-graded algebra $A$ with the associated $\hat{G}$-action~$\phi$ and assume that $|G|$ is finite and not divisible by $\charr(\mk)$. Do we have  an isomorphism $\Ext^\bullet(\Phi)\cong\Ext^{\bullet}(\Pi)$?
\item For two algebras $A$ and $B$ with $H$-actions $\phi$ and $\omega$ and a (strong) equivalence of triangulated categories $\cD^b(A)\to\cD^b(B)$ intertwining $\Phi$ and $\Omega$, do we have $\Ext^\bullet(\Phi)\cong\Ext^\bullet(\Omega)$?
\item If two $G$-graded algebras $A$ and $B$ are (strongly) gradable derived equivalent, do we have $\Ext^{\bullet}(\Pi^{(A)})\cong \Ext^{\bullet}(\Pi^{(B)})$?
\end{enumerate}

\appendix

\section{Proofs of Section~\ref{GroupAc}}\label{AppGroup}

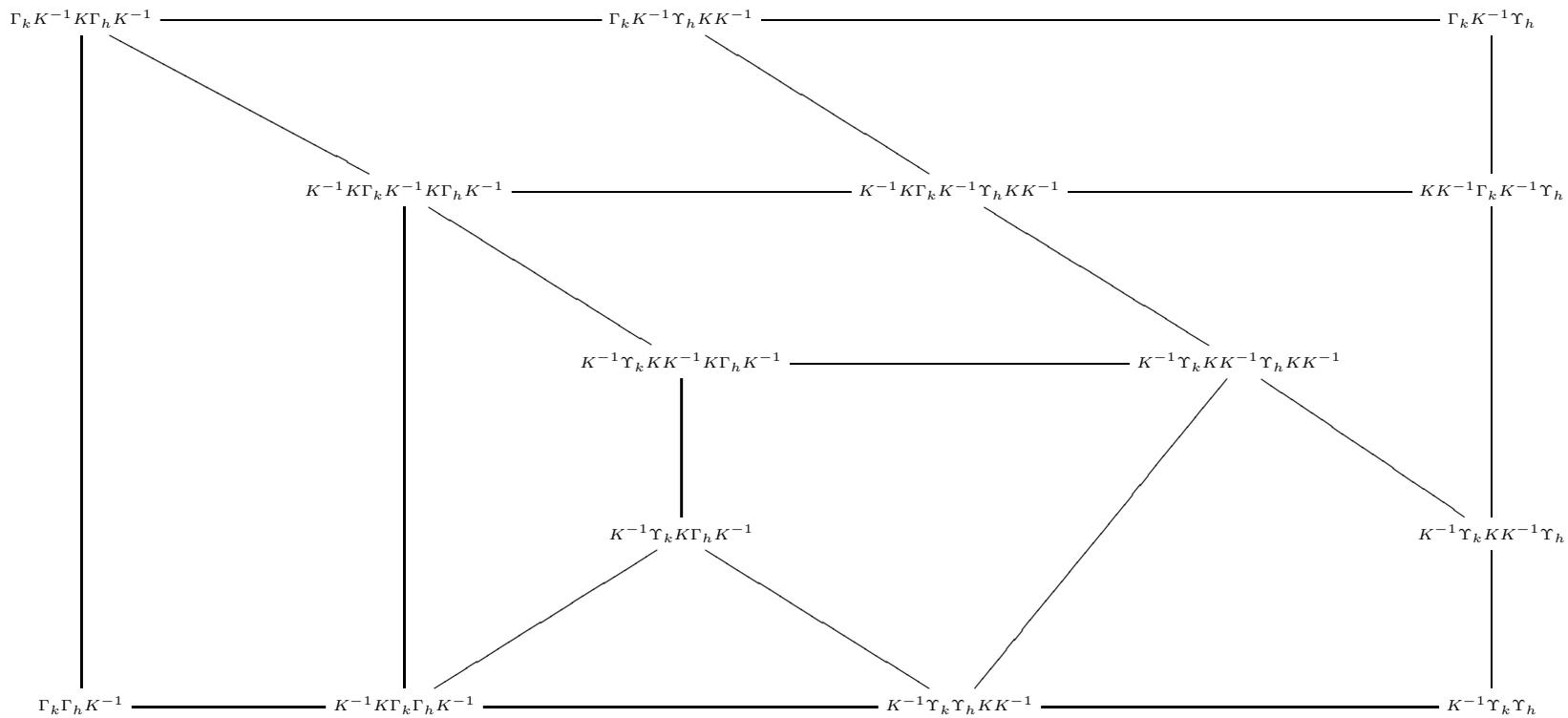
\begin{sidewaysfigure}
{\tiny
\begin{displaymath}
\xymatrix{ 
\\ \\ \\ \\ \\ \\ \\ \\ \\ \\ \\
\Gamma_kK^{-1}K\Gamma_hK^{-1}\ar@{-}[dddddddd]\ar@{-}[rrr]\ar@{-}[ddrr]
&&&\Gamma_kK^{-1}\Upsilon_hKK^{-1}\ar@{-}[rrr]&&& \Gamma_kK^{-1}\Upsilon_h\ar@{-}[dd]\\ \\
&& K^{-1}K\Gamma_kK^{-1}K\Gamma_hK^{-1}&&K^{-1}K\Gamma_kK^{-1}\Upsilon_hKK^{-1}\ar@{-}[rr]\ar@{-}[ll]\ar@{-}[uul]&&
KK^{-1}\Gamma_kK^{-1}\Upsilon_h\ar@{-}[dddd] \\ \\
&&&K^{-1}\Upsilon_kKK^{-1}K\Gamma_hK^{-1}\ar@{-}[luu]\ar@{-}[rr]&&
K^{-1}\Upsilon_kKK^{-1}\Upsilon_hKK^{-1}\ar@{-}[uul]&\\ \\
&&&K^{-1}\Upsilon_kK\Gamma_hK^{-1}\ar@{-}[uu]&&&
K^{-1}\Upsilon_kKK^{-1}\Upsilon_h\ar@{-}[dd]\ar@{-}[luu] \\ \\
\Gamma_k\Gamma_hK^{-1}\ar@{-}[rr]&&
K^{-1}K\Gamma_k\Gamma_hK^{-1}\ar@{-}[rr]\ar@{-}[ruu]\ar@{-}[uuuuuu]&&
K^{-1}\Upsilon_k\Upsilon_hKK^{-1}\ar@{-}[rr]\ar@{-}[luu]\ar@{-}[ruuuu]&&K^{-1}\Upsilon_k\Upsilon_h
}
\end{displaymath}
}
\caption{Commutative diagram in the proof of Proposition~\ref{InterInv}}\label{fig1}
\end{sidewaysfigure}

\begin{proof}[Proof of Proposition~\ref{InterInv}]
To prove this, consider the diagram given in Figure~\ref{fig1}. All edges of this diagram correspond to the
obvious pair of mutually inverse isomorphisms (given by using horizontal pre- and post-composition of
$\alpha$, $\beta$ or $\xi$ with necessary identity morphisms). Note that the vertical edge in the middle 
of the diagram is induced from either $\alpha$ or $\beta$, where equality of both options follows from the counit-unit adjunction formula $K(\alpha)\circ\beta_K=1_K$.

 The bottom triangle commutes because of 
commutativity of \eqref{eqCommFun2}. To check commutativity of all rectangles one uses 
associativity of horizontal composition and interchange law. This implies that 
the whole diagram commutes and establishes our claim.
\end{proof}

\begin{proof}[Proof of Proposition~\ref{PropCDEq}]
Let $\alpha$ denote an isomorphism of functors $F\circ F^{\minus 1}\stackrel{\sim}{\Rightarrow}\Id_{\cD}$. Using the notation of \ref{CatComm}, we have isomorphisms of functors
$$\delta_k=\Upsilon_k(\alpha)\circ \xi^k_{F^{\minus 1}}\;\,:\;\, F\circ\Gamma_k\circ F^{\minus 1}\;\stackrel{\sim}{\Rightarrow}\;\Upsilon_k.$$
We have the corresponding isomorphism vector spaces
$$\beta:\End(\Gamma)\to\End(\Upsilon),\quad\Nat(\Id,\Gamma_h)\ni\eta\mapsto\delta_h\circ F(\eta)_{F^{\minus 1}}\circ\alpha^{\minus 1}.$$
Now consider $\sigma\in \Nat(\Id,\Gamma_k)$.
Equation~\eqref{eqCommFun} implies that
$$\beta(\Gamma_k(\eta)\circ\sigma)\;=\;\Upsilon_{kh}(\alpha)\circ \left(\Upsilon_k(\xi^h)\circ\xi^{k}_{\Gamma_h}\circ F\Gamma_k(\eta)\circ F(\sigma)\right)_{F^{\minus1}}\circ\alpha^{\minus1}.$$
On the other hand, we have
$$\Upsilon_k(\beta(\eta))\circ\beta(\sigma)\;=\; \Upsilon_k\left(\Upsilon_h(\alpha)\circ (\xi^h\circ F(\eta))_{F^{\minus 1}}\circ\alpha^{\minus 1})\right)\circ \Upsilon_k(\alpha)\circ ( \xi^k\circ F(\sigma))_{F^{\minus 1}}\circ\alpha^{\minus 1}.$$
Using the definition of~$\xi^g$ shows that the above two expression agree, which shows that $\beta$ is an algebra isomorphism.
\end{proof}

\section{Evaluation on tilting modules vs tilting complexes}\label{AppB}

Consider a finite dimensional $A\in \Alg$. Recall from \cite[Section~6]{Derived1} that a {\em tilting complex} $T_\bullet$ in $\cD^b(A)$ is an object 
in $\cD^b(A)$ such that
\begin{itemize}
\item $\Hom_{\cD^b(A)}(T_\bullet,T_\bullet[j])=0,\qquad\mbox{ for all }\quad j\not=0$;
\item $\add(T_\bullet)$ generates $\cD^b(A)$ as a triangulated category.
\end{itemize}
Clearly, the image $F({}_AA)$ for any derived equivalence $F:\cD^B(A)\to\cD^b(B)$ is a tilting complex in $\cD^b(B)$. 

By definition, a {\em tilting module} is a tilting complex contained in one position. It follows by definition that, for a tilting module $T$ in $A$-mod and an arbitrary module~$M$ in $A$-mod, there exists a bounded complex
$$\cdots\to X_{-1}\to X_0\to X_1\to \cdots,$$
with $X_i\in \add(T)$ and such that the homologies $H_i(X_\bullet)$ are zero when $i\not=0$ and isomorphic to~$M$ when $i=0$.

\subsection{Faithful evaluation on tilting modules}

\begin{lemma}\label{LemEvT}
Let $T$ be a tilting module in $A${\rm-mod}.
\begin{enumerate}[(i)]
\item Let $F^1,F^2$ be exact endofunctors on $A${\rm -mod}, with $\eta\in\Nat(F^1,F^2)$. If $\eta_T=0$, then $\eta=0$.
\item Assume that $A$ is $G$-graded and $T$ admits a graded lift, which we denote by $T$ again. Let $F^1,F^2$ be exact endofunctors on $A${\rm -gmod}, with $\eta\in\Nat(F^1,F^2)$. If $\eta_{\Pi_gT}=0$, for all $g\in G$, then $\eta=0$.
\end{enumerate}
\end{lemma}
Both claims are special cases of the following obvious general principle.
\begin{lemma}\label{LemAbstract}
Let $F^1,F^2$ be exact endofunctors of an abelian category $\cC$, with $\eta\in\Nat(F^1,F^2)$. Assume that $\cC$ has a set $S$ of objects such that any object in $\cC$ is a subquotient of a finite direct sum of objects in $S$. Then $\eta=0$ if and only if $\eta_X=0$, for all $X\in S$.
\end{lemma}

\subsection{Non-faithful evaluation on tilting complexes}\label{AppB2}
We give an example which shows that Lemma~\ref{LemEvT}(i) does not naturally extend to tilting complexes.

Let $A$ be the hereditary path algebra of the quiver
$$\xymatrix{ 1\ar[r]^a& 2\ar[r]^b&3}.$$
We denote the identity path at $i$ by $e_i$. For a vertex $i$, we denote by $L_i$
the corresponding simple $A$-module, by $P_i$ the projective cover of $L_i$, and by
$I_i$ the injective envelope of $L_i$.

Consider the complex $C_\bullet$ given by
$$0\to P_2\to I_2\to 0,$$
where  $P_2$ is in position zero and the middle morphism is not zero.
It is easily checked that $T_\bullet:= P_3\oplus P_2\oplus C_\bullet$ is a tilting complex.

Now we consider the bimodules
$$X_1=Ae_3\otimes_{\mk} e_1A\quad\mbox{and}\quad X_2=Ae_1\otimes_{\mk}e_3A$$
and corresponding exact functors
$$F^1=X_1\otimes_A-\quad\mbox{and}\quad F^2=X_2\otimes_A-$$
on $A$-mod.
We have a natural transformation $\eta:F^1\Rightarrow F^2$ corresponding to the morphism $X_1\to X_2$, which maps the simple bimodule $X_1$ to the socle of $X_2$, this is the morphism
$$X_1\to X_2,\quad e_3\otimes e_1\mapsto ba\otimes ba.$$

Observe that, for any $A$-module $M$, we have $F^1M=0$ unless $[M:L_1]\not=0$ and $F^2M=0$ unless $[M:L_3]\not=0$. It thus follows easily that $\eta_M=0$ unless $M=P_1$. Since $P_1$ does not appear inside $T_\bullet$, it follows that $\eta_{T_\bullet}=0$, for $\eta\in\Nat(F^1_\bullet,F^2_\bullet)$ induced from the natural transformation $F^1\Rightarrow F^2$ considered above.

Hence, the composition
$$\Nat(F^1,F^2)\;\to\;\Nat(F^1_\bullet,F^2_\bullet)\;\to\; \Hom_{\cD^b(A)}(F^1_\bullet T_\bullet,F^2_\bullet T_\bullet),$$
is not injective.

\subsection*{Acknowledgement}
K.C. is supported by Australian Research Council Discover-Project Grant DP140103239. 
V.M. is supported by the Swedish Research Council and the G{\"o}ran Gustafssons Foundation.
We thank Maria Gorelik for discussions which motivated this paper and Martin Herschend for discussions leading to the example in Appendix~\ref{AppB2}.

\vspace{2mm}

\noindent
KC: School of Mathematics and Statistics, University of Sydney, NSW 2006, Australia;
E-mail: {\tt kevin.coulembier@sydney.edu.au} 
\vspace{2mm}

\noindent
VM: Department of Mathematics, University of Uppsala, Box 480, SE-75106, Uppsala, Sweden;
E-mail: {\tt  mazor@math.uu.se}
\date{}

\end{document}